\documentclass{amsart}
\usepackage[utf8]{inputenc}
\usepackage[english]{babel}
\usepackage{graphicx}
\usepackage{array}
\usepackage{amsmath}
\usepackage{amsthm}
\usepackage{amsfonts} 
\usepackage[T1]{fontenc}
\usepackage{graphicx}
\usepackage[bottom=3cm, left=2.5cm, right=2.5cm, top=2.5cm]{geometry}
\usepackage{subfigure} 
\usepackage{color} 
\usepackage{cite} 

\newtheorem{lemat}{Lemma}[section]
\newtheorem{tw}{Theorem}[section]
\newtheorem{ass}{Assumption}[section]
\newtheorem{defi}{Definition}[section]
\newtheorem*{uwaga}{Remark}
\numberwithin{equation}{section}

\def\vc#1{\boldsymbol{#1}}
\def\ten#1{\boldsymbol{#1}}

\newcommand{\dxdt}{\,{\rm{d}}x\,{\rm{d}}t}
\newcommand{\dx}{\,{\rm{d}}x}

\newcommand{\ds}{\,{\rm{d}}s}
\newcommand{\dt}{\,{\rm{d}}t}
\newcommand{\dtau}{\,{\rm{d}}\tau}
\newcommand{\braket}[1]{(\!(#1)\!)}

\def\div{\rm{div\,}}

\title{Thermo-visco-elasticity for Norton-Hoff-type models}

\author[P. Gwiazda]{Piotr Gwiazda}
\address{Institute of Applied Mathematics, University of Warsaw, ul. Banacha 2, 02-097 Warsaw, Poland}
\email{pgwiazda@mimuw.edu.pl}

\author[F. Z. Klawe]{Filip Z. Klawe}
\address{Institute of Applied Mathematics, University of Warsaw, ul. Banacha 2, 02-097 Warsaw, Poland}
\email{aswiercz@mimuw.edu.pl}

\author[A. \'Swierczewska-Gwiazda]{Agnieszka \'Swierczewska-Gwiazda}
\address{Institute of Applied Mathematics, University of Warsaw, ul. Banacha 2, 02-097 Warsaw, Poland}
\email{fzklawe@mimuw.edu.pl}

\begin{document}

\begin{abstract}
Our research is directed to a quasi-static evolution of the thermo-visco-elastic model. We assume that the material is subject to two kinds of  mechanical deformations: elastic and inelastic. Moreover, our analysis captures the influence of the temperature on the visco-elastic properties of the body. 
The novelty of the paper is the consideration of the thermodynamically complete model to describe this kind of phenomena related with  a hardening rule of Norton-Hoff type. We provide the proof of   existence of solutions to thermo-visco-elastic model in a simplified setting, namely  the thermal expansion effects are neglected. Consequently, the coupling between the temperature and the displacement occurs only in the constitutive function for the evolution of the visco-elastic strain. 

\end{abstract}

\maketitle

\section{Introduction}

We are aiming to describe the response of thermo-visco-elastic material to applied external forces and the heat flux through the boundary. The  system of equations capturing  the displacement,  temperature and visco-elastic strain of the body is 
a consequence of physical principles such as balance of momentum and balance of energy, cf. \cite{GreenNaghdi, LandauLifshitz}, see  also  \cite{GKSG}. The equations are  complemented by the  constitutive relation for the Cauchy stress tensor and the constitutive equation for the evolution of the visco-elastic strain tensor.
Although we treat the case where the thermal expansion is negligible, but the changes of temperature affect the visco-elastic properties of the considered material. We shall observe it in the appearance of the temperature-dependent constitutive relation in the evolution equation for the visco-elastic strain tensor. 
%

We  assume that the  body $\Omega \subset \mathbb{R}^3$ is an open bounded set with a $C^1$ boundary and moreover, the body  is homogeneous in space. The material 
undergoes
two kinds of  deformations:  elastic and visco-elastic. By the first type we understand the deformations which are reversible and the second ones are  irreversible. 
The problem is captured by the following system
\begin{equation}
\left\{
\begin{array}{rclr}
- \div \ten{T} &=& \vc{f} & \mbox{in } \Omega\times(0,T),
\\
\ten{T} &=& \ten{D}(\ten{\varepsilon}(\vc{u}) - \ten{\varepsilon}^{\bf p} ) & \mbox{in } \Omega\times(0,T),
\\
\ten{\varepsilon}^{\bf p}_t &=& \ten{G}(\theta,\ten{T}^d) & \mbox{in } \Omega\times(0,T),
\\
\theta_t - \Delta\theta &=& \ten{T}^d:\ten{G}(\theta,\ten{T}^d) & \mbox{in } \Omega\times(0,T),
\end{array}
\right.
\label{full_system_2}
\end{equation}
which describes the quasi-static evolution of the displacement of the material $\vc{u}:\Omega\times\mathbb{R}_+\rightarrow \mathbb{R}^3$, the temperature of the material $\theta:\Omega\times\mathbb{R}_+\rightarrow\mathbb{R}_+$ and the visco-elastic strain tensor $\ten{\varepsilon}^{\bf p}:\Omega\times\mathbb{R}_+\rightarrow \mathcal{S}^3_d$.
We denote by $\mathcal{S}^3$ the set of symmetric $3 \times 3$-matrices with real entries and by $\mathcal{S}^3_d$ a subset of $\mathcal{S}^3$ which contains traceless matrices. By $\ten{T}^d$ we mean the deviatoric part (traceless) of the tensor $\ten{T}$, i.e. $\ten{T}^d=\ten{T}-\frac{1}{3}tr(\ten{T})\ten{I}$, where $\ten{I}$ is the identity matrix from $\mathcal{S}^3$. Additionally, $\ten{\varepsilon}(\vc{u})$ denotes the symmetric part of the gradient of the displacement~$\vc{u}$, i.e. $\ten{\varepsilon}(\vc{u})=\frac{1}{2}(\nabla\vc{u} + \nabla^T\vc{u})$.  The volume force is denoted by $\vc{f}:\Omega\times\mathbb{R}_+\rightarrow \mathbb{R}^3$.


The visco-elastic strain tensor is described by the evolutionary equation with prescribed constitutive function $\ten{G}(\cdot,\cdot)$. 
The function $\ten{G}(\cdot,\cdot)$ is assumed to be monotone and to satisfy polynomial growth and coercivity conditions. 
\begin{ass}
The function $\ten{G}(\theta,\ten{T}^d)$ is  continuous with respect to   $\theta$ and $\ten{T}^d$ and satisfies for $p\ge 2$
the following conditions:
\begin{itemize}
\item[a)] $(\ten{G}(\theta,\ten{T}^d_1)-\ten{G}(\theta,\ten{T}^d_2)):(\ten{T}^d_1-\ten{T}^d_2) \geq 0$, for all $\ten{T}_1^d,\ten{T}_2^d \in \mathcal{S}^3_d$ and $\theta\in\mathbb{R}_+$;
\item[b)] $|\ten{G}(\theta,\ten{T}^d)| \leq C(1 + |\ten{T}^d|)^{p-1}$, where $\ten{T}^d\in\mathcal{S}^3_d$, $\theta\in\mathbb{R}_+$;
\item[c)] $\ten{G}(\theta,\ten{T}^d):\ten{T}^d \geq \beta |\ten{T}^d|^p$, where $\ten{T}^d\in\mathcal{S}^3_d$, $\theta\in\mathbb{R}_+$,
\end{itemize}
\label{ass_G}
where  $C$ and $\beta$ are positive constants,  independent of the temperature $\theta$.
\end{ass}

We complete the considered problem by formulating  the  initial conditions
\begin{equation}
\left\{
\begin{array}{rcl}
\theta(x,0)&=&\theta_0(x), 
\\
\ten{\varepsilon}^{\bf p}(x,0)&=&\ten{\varepsilon}^{\bf p}_0(x),
\end{array}
\right.
\label{init_0}
\end{equation}
in $\Omega$ and boundary conditions
\begin{equation}
\left\{
\begin{array}{rcl}
\vc{u}&=&\vc{g}, \\
\frac{\partial \theta}{\partial \vc{n}}&=&g_{\theta},
\end{array}
\right.
\label{boun_0}
\end{equation}
on $\partial\Omega\times(0,T)$.

The properties of the material under consideration determine the choice of the function $\ten{G}$. 
Such a framework includes  the classical 
Norton-Hoff model, cf.~\cite{CheAl}, which we shall briefly discuss in a sequel.
There are various different relations considered,
e.g. 
\begin{itemize}
\item
Bodner-Partom model \cite{Bartczak, MANA:MANA5, MMA:MMA802}:
\begin{equation}
\begin{split}
\ten{G}(\theta,\ten{T}^d) &= \mathcal{G}\left(\frac{\left\{|\ten{T}^d| + \beta(\theta) \right\}^+}{y} \right)\frac{\ten{T}^d}{|\ten{T}^d|}
,
\\
y_t & = \gamma(y) \mathcal{G}\left(\frac{|\ten{T}^d|}{y} \right) |\ten{T}^d| - A \delta(y),
\end{split}
\end{equation}
where $y : \Omega \times \mathbb{R}_+ \to \mathbb{R}_+$ describes the isotropic hardening of the metal, $\{\cdot\}^+$ stands for the positive part of $\{\cdot\}$, $\gamma : \mathbb{R}_+ \supset D (\gamma) \to \mathbb{R}_+$ and $\delta : \mathbb{R}_+ \supset D(\delta) \to \mathbb{R}_+$ are given functions and $A$ is a positive constant. Moreover, functions $\mathcal{G}(\cdot)$, $\gamma(\cdot)$, $\delta(\cdot)$ and $\beta(\cdot)$ fulfill some specific properties.

\item Mr\'{o}z model \cite{GKSG, brokate, Homberg200455}:
\begin{equation}
\ten{G}(\theta,\ten{T}^d) = g(\theta)\ten{T}^d,
\end{equation}
where $g:\mathbb{R}_+\to \mathbb{R}_+$ is a continuous function.
\item Prandtl-Reuss model with linear kinematic hardening \cite{ChR}
\begin{equation}
\begin{split}
\ten{\varepsilon}^{\bf p}_t &\in \partial I_{K(\theta)}(\ten{T} - \alpha \ten{\varepsilon}^{\bf p}),
\end{split}
\end{equation}
where $I_{K(\theta)}$ is the indicator function of the closed and convex subset $K(\theta) = \{\ten{T}\in\mathcal{S}^3: |\ten{T}^d|\leq k-\theta\}$ and $\alpha, k>0$ are material parameters.  Furthermore, $\partial I_{K(\theta)}$ is a subdifferential of the function $I_{K(\theta)}$.
\end{itemize}
For further  examples of constitutive relations (e.g. classical Maxwell model, models proposed by Chaboche, Hart, Miler, Bruhns and many others) we refer to \cite[Chapter 2.2]{Alber}.

Our motivation for current considerations were the results of 
Alber and Chełmiński \cite{CheAl} and  of H\"{o}mberg \cite{Homberg200455}. 
In \cite{CheAl} the authors considered the quasi-static visco-elasticity\footnote{The authors used the notion {\it visco-plasticity} which is sometimes also applied in the literature to capture the appearance of irreversible deformations.}
 models with Norton-Hoff constitutive function, namely of the power-law type
\begin{equation}
\ten{G} = c |\ten{T}|^{p-1} \ten{T}
\label{eq:norton-hoff}
\end{equation}
with $p>2$.  The parameter $c$ was either assumed to be a positive constant or dependent on an additional relaxation parameter described by a separate  equation. 
  The scheme of the proof in \cite{CheAl} was to formulate the problem in a way that it fits to the abstract theory of maximal monotone operators, cf.~\cite{Barbu}. In the current paper we include the thermal effects of the process through
 the dependence of  the constitutive function $\ten{G}$ on the temperature. This dependence  obstructs  following the same scheme and
 requires different approach.

 Furthermore, we assume that  $\ten{G}(\theta,\cdot)$ depends only on the deviatoric part of the Cauchy stress tensor and its range is the set of traceless matrices. The last assumption, together with the fact that $\ten{\varepsilon}^{\bf p}_0(x)$ is traceless, provides that also $\ten{\varepsilon}^{\bf p}$ is  traceless. Vanishing of the trace of the deformation tensor corresponds to preserving the volume of the material.  Indeed, the volume change is associated only with the elastic response of the material, and the plastic response is essentially incompressible, cf.~\cite{Gurtin}. The dependence of $\ten{G}(\theta,\cdot)$ only on $\ten{T}^d$ is essential to maintain the coercivity of the model. Once we know that the range of $\ten{G}$ is 
 ${\mathcal S}_d^3$, then even for the isothermal process, namely the case of $\ten{G}=\ten{G}(\ten{T}) $
 we observe that  $\ten{G}(\ten{T}):\ten{T}=\ten{G}(\ten{T}):\ten{T}^d$. Then e.g. taking as $\ten{T}$ the identity matrix we immediately see that  $\ten{G}(\ten{I}):\ten{I}^d=0$.
Let us now comment on the technical consequences of this assumption.  Contrary to the proof of  Alber and Chełmiński, where they showed that $\ten{T}$ belongs to $L^p(0,T,L^p(\Omega,\mathcal{S}^3))$ for $p\geq 2$, the estimates conducted in the current situation provide only  that $\ten{T}$ belongs to $L^2(0,T,L^2(\Omega,\mathcal{S}^3))$.


H\"{o}mberg in \cite{Homberg200455} considered more general physical phenomena including  the electro-magnetic effects. The changes of temperature influenced the  concentration of different phases of materials and
this dependence was prescribed by some general operator $\mathcal{P}[\cdot]$ having {\it good} properties.
Then the constitutive function describing the evolution of visco-elastic strain depends no more on the temperature, but on these concentrations. Moreover it is linear with respect to the deviatoric part of the Cauchy stress tensor, namely corresponds to   the Mr\'{o}z model. 
The similarities with our approach are related with the construction of the approximated problem, namely by the truncation of  the terms which appear on the right-hand side of the heat equation and are only integrable. The method also follows the framework of Boccardo and Gallou\"et. Nevertheless, because of the different structure of the problem, H\"{o}mberg can show the strong convergence of the approximated sequence of the Cauchy stress tensor.  For the concept of showing this strong convergence observe that in the case of linear Mr\'oz relation, and in fact also in the case of Norton-Hoff relation \eqref{eq:norton-hoff}, the stronger condition than monotonicity holds, namely  the uniform monotonicity condition 
$$(\ten{G}(\theta,\ten{T}^d_1)-\ten{G}(\theta,\ten{T}^d_2)):(\ten{T}^d_1-\ten{T}^d_2) \geq c | \ten{T}^d_1-\ten{T}^d_2|^p \mbox{ for all } \ten{T}_1^d,\ten{T}_2^d \in \mathcal{S}^3_d \mbox{ and }\theta\in\mathbb{R}_+.$$
For the proof see e.g.~\cite{maleknecas}. 

The studies on the Mr\'oz model presented in~\cite{GKSG} essentially used the strong monotonicity of the function 
$\ten{G}$ in the second variable. The existence proof used the methods developed in~\cite{GwSw2005, Sw2006, ChGw2007} arising from the tools of Young measures. In the present setting none of the assumptions of strong nor uniform monotonicity are needed. We only assume monotonicity of $\ten{G}$.

Following Bartczak \cite{Bartczak}, Chełmiński \cite{MMA:MMA802}, Chełmiński and Racke \cite{ChR}, Duvaut and J.L. Lions \cite{duvautLions}, Johnson \cite{johnson1,johnson2}, Ne\v{c}as and Hlav\'{a}\v{c}ek \cite{NH}, Suquet \cite{suquet1,suquet2,suquet3}, Temam \cite{temam1,temam2} and many others, we study the quasi-static evolution, i.e. the evolution, which is slow and we neglect the acceleration term in the equation for balance of momentum. Moreover, we consider the model with infinitesimal displacement. In a consequence, the dependence between the Cauchy stress tensor and the symmetric gradient of displacement is linear (generalized Hooke's law, for more details see \cite{NH} or \cite{rajagopal}). 
Much of the approaches involve the models that are purely mechanical, namely concern
the theory of inelastic and infinitesimal deformations  with the nonlinear inelastic constitutive relation  of monotone type, however neglect all thermal influences, see~\cite{Alber} and also \cite{CHelG2,MANA:MANA5,MMA:MMA844, MMA:MMA802}. On the other hand,
the mathematical analysis of linear thermo-elasticity is also a classical, well understood topic, cf. \cite{JR}, contrary to an analysis of thermo-inelastic models. By the thermo-inelastic models we mean the systems consisting of balance of momentum for kind of inelastic deformation and the equation for an evolution of the temperature. In the equation for balance of momentum for inelastic deformation the stress is not proportional to the strain, i.e. there appear term which absorbs the mechanical energy. There are only some results for special models or for simplified models in the literature \cite{Bartczak, BR, ChR}.

If we introduce thermal effects into various purely mechanical models, then  the right hand side of the heat equation 
(the product $\ten{T}^d:\ten{G}(\theta,\ten{T}^d)$) turns out to be only an integrable function. 
In such a case  the standard energy methods fail and one needs to search for  more delicate tools. 
Using the Boccardo and Gallou\"{e}t \cite{Boccardo} approach to prove the existence of solutions to the heat equation the essential point is to use the truncation of the solution as a test function. This is however difficult to combine with a classical 
Galerkin method as the truncation of a function may no longer be a linear combination of the functions from the Galerkin basis. 
Therefore we appeal to non-standard energy methods, such as two-level Galerkin approximation, see also  \cite{PhDBulicek,BFM,Bull}. The new difficulty which arises here is the construction of the appropriate basis for approximation of the strain tensor $\ten{\varepsilon}^{\bf p}$, for details see Appendix~\ref{B}.

All functions appearing in this paper are the  functions of position $x$ and time $t$. We often omit the variables of the function and write $\vc{u}$ instead of $\vc{u}(x,t)$. All of the computation are conducted  in Lagrangian coordinates. In view of  the fact that the displacement is small, the stress tensor in the Lagrangian coordinates is approximated by the stress tensor in Eulerian coordinates. This is a standard way of considering the inelastic models, for more details see \cite[Chapter 13.2]{temammiranville}.

Before we formulate the definition of weak solutions and state the main result of the paper let us introduce the notation  
$W^{1,p'}_{\vc{g}}(\Omega,\mathbb{R}^3):=\left\{\vc{u} \in W^{1,p'}(\Omega,\mathbb{R}^3): \vc{u}=\vc{g} \mbox{ on } \partial\Omega \right\}$. 

\begin{defi}
Let $p\geq 2$, $q<\frac{5}{4}$ and $s\in\mathbb{R}$ be large enough. The triple of functions 
\begin{equation}
\begin{split}
\vc{u}&\in L^{p'}(0,T,W^{1,p'}_{\vc{g}}(\Omega,\mathbb{R}^3))
\\
\ten{T}&\in L^2(0,T,L^2(\Omega,\mathcal{S}^3))
\end{split}
\nonumber
\end{equation}
and 
\begin{equation}
\theta\in L^q(0,T,W^{1,q}(\Omega))\cap C([0,T],W^{-s,2}(\Omega))
\nonumber
\end{equation}
is a weak solution to the system \eqref{full_system_2} if
\begin{equation}
\begin{split}
\int_0^T\int_{\Omega}\ten{T}:\nabla\vc{\varphi} \dxdt 
&= \int_0^T\int_{\Omega}\vc{f}\cdot \vc{\varphi} \dxdt ,
\end{split}
\end{equation}
where 
\begin{equation}
\ten{T}=\ten{D}(\ten{\varepsilon}(\vc{u}) - \ten{\varepsilon}^{\bf p}),
\end{equation}
and
\begin{equation}
\begin{split}
-\int_0^T\int_{\Omega} \theta\phi_t \dxdt -
\int_{\Omega} \theta_0(x)\phi(0,x) \dx \qquad \qquad \qquad & \\ + 
\int_0^T\int_{\Omega}  \nabla\theta\cdot\nabla\phi  \dxdt -
\int_0^T\int_{\partial\Omega}g_{\theta}\phi  \dxdt &= 
\int_0^T\int_{\Omega} \ten{T}^d:\ten{G}(\theta,\ten{T}^d)\phi \dxdt,
\end{split}
\end{equation}
holds for every test function $\vc{\varphi}\in C^{\infty}([0,T],C^{\infty}_c(\Omega,\mathbb{R}^3))$ and $\phi\in C^{\infty}_c([-\infty,T),C^{\infty}(\Omega))$. Furthermore, the visco-elastic strain tensor can be recovered from the equation on its evolution, i.e.
\begin{equation}
\ten{\varepsilon}^{\bf p}(x,t) = \ten{\varepsilon}^{\bf p}_0(x) + \int_0^t \ten{G}(\theta(x,\tau),\ten{T}^d(x,\tau)) \dtau,
\end{equation}
for a.e. $x\in\Omega$ and $t\in [0,T)$. Moreover, $\ten{\varepsilon}^{\bf p} \in W^{1,p'}(0,T,L^{p'}(\Omega,\mathcal{S}^3_d))$.

\end{defi}

\begin{tw}
Let $p\geq 2$ and let initial conditions satisfy $\theta_0 \in L^1(\Omega)$, $\ten{\varepsilon}^{\bf p}_0\in L^2(\Omega,\mathcal{S}^3_d)$, boundary conditions satisfy $\vc{g}\in L^p(0,T, W^{1-\frac{1}{p},p}(\partial\Omega,\mathbb{R}^3))$, $g_{\theta}\in L^2(0,T,L^2(\partial\Omega))$ and volume force $\vc{f}\in L^p(0,T,W^{-1,p}(\Omega,\mathbb{R}^3))$ and function $\ten{G}(\cdot,\cdot)$ satisfy the  Assumption \ref{ass_G}. Then there exists a weak solution to system \eqref{full_system_2}. 
\label{thm:main2}
\end{tw}

\begin{uwaga}
There is nothing about the uniqueness of solutions in  Theorem \ref{thm:main2}. Using Boccardo and Gallou\"{e}t approach to the heat equation we obtain  the existence of $\theta$ only in the space $L^q(0,T,W^{1,q}(\Omega))$ for all $q<\frac{5}{4}$, see the Appendix. The lack of uniqueness of the temperature implies the  lack of uniqueness of the solution to the  whole system.
In order to find the class of solutions providing both existence and uniqueness one should consider the renormalized solutions to the heat equation, see \cite{BlanchardMurat,Blanchard}. However, the existing theory concerns mostly the homogeneous Dirichlet boundary-value problems. 
\end{uwaga}

The rest of the paper is organized as follows:  Section \ref{druga} is mostly dedicated to physical aspects of the problem. Therefore in Section \ref{sec:model} we introduce the complete  model and present the assumptions which brought us to the simplified setting. Then in Section \ref{comp} we concentrate on physical justification of the model after simplifications. Section \ref{sec:3} is only a technical part that prepares us to the proof of the main theorem, namely  we transform the problem to a homogeneous boundary-value problem.  
The whole 
Section \ref{sec:proof} is devoted to the proof of Theorem \ref{thm:main2}. The subsequent subsections correspond to the steps of the proof such as  existence of the approximate solutions,  boundedness of the approximate solution and the behaviour of the  energy of the system. Finally we pass to the limit in the Galerkin approximations.

We complete the introduction by  introducing the notation. 
As a result of  integration $\int_{t_1}^{t_2}\frac{d g}{d t} dt$ we write $g|_{t_1}^{t_2} $ which is equal to $g(t_2) - g(t_1)$. Furthermore, we denote by $L^p(\Omega)$ standard Lebesgue spaces, for $k,m\in\mathbb{N}$ and $1\leq p,q \leq \infty$, by $W^{k,p}(\Omega)$ the Sobolev spaces, by $W^{\frac{m}{k},p}(\partial\Omega)$ the fractional order Sobolev space and by $L^p(0,T,L^q(\Omega))$ Bochner spaces, by $C(K)$ continuous functions on $K$, by $C_c^{\infty}(K)$ compactly supported smooth functions on $K$.

\section{The physical model. Motivations and simplifications.}\label{druga}
We will start the  current section with formulating the full system describing the evolution of visco-elastic body including  thermal effects. Subsequently we describe the assumptions that were made due to simplify the system and motivate considering equations \eqref{full_system_2}. The second part concerns the issue of thermodynamical completeness of the considered system. This part essentially follows \cite{GKSG}. However, since this is an important argument for choosing this model, we include the main steps for completeness. In the last subsection we include the technical step which allows to reduce the problem to homogeneous boundary-value problem. 

\subsection{Origin of the model problem}
\label{sec:model}
Let us consider the system of equations in the bounded domain $\Omega\subset\mathbb{R}^3$ with a $C^1$ boundary $\partial \Omega$
\begin{align}
\varrho\vc{u}_{tt} - \div \ten{\sigma} &= \vc{f} & \mbox{in } \Omega\times(0,T),
\label{full_system1}
\\
\ten{\sigma} &= \ten{T} - \alpha(\theta-\theta_R)\ten{I} & \mbox{in } \Omega\times(0,T),
\label{full_system2}
\\
\ten{T} &= \ten{D}(\ten{\varepsilon}(\vc{u}) - \ten{\varepsilon}^{\bf p} ) & \mbox{in } \Omega\times(0,T),
\label{full_system3}
\\
\ten{\varepsilon}^{\bf p}_t &= \ten{G}(\theta,\ten{T}^d) & \mbox{in } \Omega\times(0,T),
\label{full_system4}
\\
\theta_t - \kappa\Delta\theta +\alpha(\theta-\theta_R)\div \vc{u}_t &= \ten{T}^d:\ten{G}(\theta,\ten{T}^d) + r & \mbox{in } \Omega\times(0,T).
\label{full_system5}
\end{align}

Derivation of this system was presented in \cite{GreenNaghdi}, \cite{GKSG} and \cite{LandauLifshitz}. The equation \eqref{full_system1} describes the balance of momentum. Equations \eqref{full_system2} and \eqref{full_system3} prescribe the constitutive relation for the Cauchy stress tensor and \eqref{full_system4} presents the constitutive relation for the evolution of the visco-elastic strain tensor. Finally, \eqref{full_system5} stands for the balance of energy.

The function $\ten{\sigma}:\Omega\times\mathbb{R}_+\rightarrow \mathcal{S}^3$ is the Cauchy stress tensor. The Cauchy stress tensor can be divided into two parts: mechanical and thermal. The mechanical part is $\ten{T}=\ten{D}(\ten{\varepsilon}(\vc{u}) - \ten{\varepsilon}^{\bf p})$, where the operator $\ten{D}:\mathcal{S}^3\rightarrow\mathcal{S}^3$ is linear, positively definite and bounded. 
Assuming that $\Omega$ is a homogeneous material, the operator $\ten{D}$ is a four-index matrix, i.e. $\ten{D}=\left\{d_{i,j,k,l}\right\}_{i,j,k,l=1}^3$ and the following equalities  hold
\begin{equation}
d_{i,j,k,l} = d_{j,i,k,l},
\quad
d_{i,j,k,l} = d_{i,j,l,k}
\quad
\mbox{and}
\quad
d_{i,j,k,l} = d_{k,l,i,j}
\quad
\forall i,j,k,l=1,2,3 .
\end{equation}

The evolution of the visco-elastic strain tensor $\ten{\varepsilon}^{\bf p}$ is governed by the constitutive relation   $\ten{G}:\mathbb{R}_+\times\mathcal{S}^3_d\rightarrow\mathcal{S}^3_d$.
 The visco-elastic strain tensor $\ten{\varepsilon}^{\bf p}=(\ten{\varepsilon}^{\bf p})^d$ is traceless if $\ten{\varepsilon}^{\bf p}_0$ is traceless. The temperature $\theta_R$ is the reference temperature.
The function $r:\Omega\times\mathbb{R}_+\to\mathbb{R}_+$ describes a given density of heat sources, $\kappa:\Omega\times\mathbb{R}_+\to\mathbb{R}_+$ is  the material’s conductivity, which in the case of homogeneous materials is a positive constant, $\varrho$ is the constant density of the body. 
Moreover,  $\alpha$  describes the thermal expansion of the body. We will study the simplified situation, namely under the following assumptions 
%
\begin{ass}
We consider only the problem with small inertial force, i.e. $\varrho\vc{u}_{tt}=0$.
\end{ass}

\begin{ass}
 We assume that $\alpha=0$, i.e. the considered material is not subject to 
 the  thermal expansion. 
\end{ass}

The fact of neglecting the acceleration term implies that the system of equations may be supplemented only by the initial conditions
\eqref{init_0}.
%
Moreover, we complete the system with boundary conditions \eqref{boun_0}. 
Using the Dirichlet condition for the displacement means that we control the shape of the body, and by using the Neumann condition for the temperature we control the flow of the energy through the boundary. 

There are various simplifications that are proposed due to provide the mathematical analysis of the system. 
In the linear thermoelasticity, the term connected with thermal expansion in the heat equation is approximated by a linear one, i.e. $\alpha(\theta - \theta_0 )\div \vc{u}_t\approx \alpha_0 \div{} \vc{u}_t$ with the argumentation that the temperature $\theta$ in the considered process is close to the reference temperature, cf. Bartczak \cite{Bartczak}, Chełmiński and Racke \cite{ChR}. 
 
From the point of view of mathematical techniques used in the linear theory, such as e.g. linear semigroup theory, this approximation seems to be accurate. Unfortunately, in a consequence of this procedure one obtains the model which is not consistent with the physical principles.  
Our simplification follows different way,  we consider the case where no thermal expansion appears, hence $\alpha=0$. In the proceeding  section we discuss in detail the issue of thermodynamical completeness of the system after the simplifications. Finally, we also formulate the last assumption.

\begin{ass}
We assume that there are no heat sources in the system, hence $r\equiv 0$. 
The material's conductivity $\kappa$ is  for simplicity equal to  $1$.
\end{ass}
Taking into account the above conditions we obtain the considered system \eqref{full_system_2}.

\subsection{Thermodynamical completeness}\label{comp}

The purpose of the current section is to underline  the physical advantages of the considered system. The  assumptions used in the construction of the simplified model  do not effect the  loss of physical properties, i.e. the system \eqref{full_system_2} still conserves the energy, the temperature is positive and there exists a function of state, namely the entropy, which has a positive rate of production. We shall say that the system is thermodynamically complete if these properties are satisfied.
 In \cite{GKSG} we showed the thermodynamical completeness of   the system \eqref{full_system1}--\eqref{full_system5} in the case it is  isolated,  i.e.  $\vc{f}=0$, we assume homogeneous boundary values and there are no  heat sources ($r=0$). 
All of the calculation in this section are formal.

\noindent
{\it Conservation of total energy} 

\noindent
In the first step we intend to show that the global energy is preserved. Multiplying the first equation of system \eqref{full_system_2} by $\vc{u}_t$ and integrating over an arbitrary set $\mathcal{O}\subset\Omega$, we obtain
\begin{equation}
-
\int_{\mathcal{O}}{\div}\ten{T}\cdot\vc{u}_t \dx 
= 0
\end{equation}
and hence
\begin{equation}
\int_{\mathcal{O}}\ten{T}:\nabla\vc{u}_t \dx 
- \int_{\partial\mathcal{O}}\ten{T}\vc{n}\cdot\vc{u}_t \ds 
= 0.
\label{rownanie_powyzej}
\end{equation}
We multiply the evolutionary equation for the visco-elastic strain by $\ten{T}$ and integrate over~$\mathcal{O}$. Subtracting this equation from \eqref{rownanie_powyzej} implies that
\begin{equation}
\begin{split}
\int_{\mathcal{O}}\big(\ten{T}:\nabla\vc{u}_t-
\ten{T}:\ten{\varepsilon}^{\bf p}_t\big)
 \dx 
- \int_{\partial\mathcal{O}}\ten{T}\vc{n}\cdot\vc{u}_t  \ds
= 
- \int_{\mathcal{O}}\ten{T}^d:\ten{G} \dx .
\end{split}
\end{equation}
Finally, using the symmetry of $\ten{T}$ we obtain 
\begin{equation}
\begin{split}
\frac{1}{2}\frac{d}{dt}\int_{\mathcal{O}}\ten{T}:(\ten{\varepsilon}(\vc{u}) -\ten{\varepsilon}^{\bf p}) \dx  
- \int_{\partial\mathcal{O}}\ten{T}\vc{n}\cdot\vc{u}_t  \ds
= 
- \int_{\mathcal{O}}\ten{T}^d:\ten{G} \dx .
\end{split}
\label{enrgia}
\end{equation}
Since the global energy of the set $\mathcal{O}$ is equal to $\mathcal{E}_{\mathcal{O}}(\tau)=\int_{\mathcal{O}}e(x,\tau)\dx $ and the density of the total energy is defined by $e(x,\tau)=\theta +\frac{1}{2}\ten{D}^{-1}\ten{T}:\ten{T} $, we obtain
\begin{equation}
\mathcal{E}_{\mathcal{O}}(t)=
\int_{\mathcal{O}}\theta(t) \dx +
\frac{1}{2}\int_{\mathcal{O}}\ten{T}:(\ten{\varepsilon}(\vc{u})-\ten{\varepsilon}^{\bf p})(t) \dx .
\end{equation}
Consequently equation \eqref{enrgia} may be written in the following from
\begin{equation}
\begin{split}
\frac{d}{dt}\mathcal{E}_{\mathcal{O}}(t) & = 
\frac{d}{dt}\int_{\mathcal{O}}\theta \dx  
- \int_{\mathcal{O}}\ten{T}^d:\ten{G} \dx 
+ \int_{\partial\mathcal{O}}\ten{T}\vc{n}\cdot\vc{u}_t  \ds .
\end{split}
\end{equation}
Using \eqref{full_system_2}$_4$, we obtain
\begin{equation}
\begin{split}
\frac{d}{dt}\mathcal{E}_{\mathcal{O}}(t) &= 
\int_{\mathcal{O}}\theta_t \dx
- \int_{\mathcal{O}}\theta_t \dx
+\int_{\mathcal{O}}\Delta\theta \dx
+ \int_{\partial\mathcal{O}}\ten{T}\vc{n}\cdot\vc{u}_t \ds
\\
& =
\int_{\partial\mathcal{O}}\big(\ten{T}\vc{u}_t +\nabla\theta\big)\cdot\vc{n} \ds .
\end{split}
\end{equation}
Zero external forces, homogeneous boundary conditions and no heat sources implies that $\vc{u}_t=0$ and $\nabla\theta\cdot\vc{n}=0$ on the boundary $\partial\Omega$. Therefore, the global energy $\mathcal{E}_{\Omega}$ is constant in time.








\noindent
{\it Positivity of the  temperature} 

\noindent
Let us assume that the initial temperature $\theta_0$ is positive. 
The heat equation after simplifications has a form
\begin{equation}
\theta_t - \Delta\theta  =
\ten{G}(\theta,\ten{T}^d):\ten{T}^d .
\label{eq:nee_heat}
\end{equation} 
Hence, the assumptions on the function $\ten{G}(\cdot,\cdot)$ imply that the right hand side of \eqref{eq:nee_heat} is positive, namely
\begin{equation}
\theta_t - \Delta\theta \geq 0.
\end{equation}
When the initial and boundary conditions for the temperature are positive, then the temperature $\theta$ is positive.

\noindent
{\it Entropy inequality} 

\noindent
Multiplying \eqref{eq:nee_heat} by $1/\theta$ and integrating over an arbitrary set $\mathcal{O}\subset\Omega$, we obtain
\begin{equation}
\begin{split}
\frac{d}{dt}\int_{\mathcal{O}}\ln\theta \dx
 -  \int_{\mathcal{O}}{\div}\frac{\nabla \theta}{\theta} \dx
- \int_{\mathcal{O}}\frac{|\nabla\theta|^2}{\theta^2} \dx
 = 
\int_{\mathcal{O}} \frac{\ten{G}(\theta,\ten{T}^d):\ten{T}^d}{\theta} \dx .
\end{split}
\nonumber
\end{equation}
Thus
\begin{equation}
\begin{split}
\frac{d}{dt}\int_{\mathcal{O}} \ln\theta 
 \dx
+ &\int_{\mathcal{O}}{\div}\Big(\frac{\vc{q}}{\theta}\Big) \dx
= \int_{\mathcal{O}} \frac{\ten{G}(\theta,\ten{T}^d):\ten{T}^d}{\theta} \dx
+ \int_{\mathcal{O}}\frac{|\nabla\theta|^2}{\theta^2} \dx .
\end{split}
\label{eq:ent_pow}
\end{equation}
By the properties of  the function $\ten{G}(\cdot,\cdot)$ and positivity of $\theta$,   the right hand side of \eqref{eq:ent_pow} is positive. 
Therefore, an arbitrary choice of the domain $\mathcal{O}$ implies that the inequality holds
\begin{equation}
\Big( \ln\theta\Big)_t 
+ {\div}\Big(\frac{\vc{q}}{\theta}\Big) \geq 0.
\end{equation}
 The above  relation is the so-called Clausius-Duhem inequality and it is one of the equivalent formulations of the second principle of thermodynamics. Hence, the homogeneous boundary conditions and the  definition of the  heat flux ($\vc{q}=-\nabla \theta$) implies that 
\begin{equation}
\frac{d}{dt}\int_{\Omega} \ln\theta \geq 0 .
\end{equation}
Note that $\eta(\theta) = \ln\theta$ is one of the admissible entropies for system \eqref{full_system_2} what furnishes a formal justification for the thermodynamical completeness of the model. For the situation with linearization of the term $\alpha(\theta-\theta_R){\div} \vc{u}_t$ one can show that none of the thermodynamical principles is fulfilled. 

\subsection{Transformation to a homogeneous boundary-value problem}
\label{sec:3}

Our aim is to reduce the problem to a homogeneous one. 
For this purpose we are interested in a  decoupled elastic systems and a heat equation.  The first system is subject to  the same external forces as problem \eqref{full_system_2} and both of the problems are complemented     with the same boundary conditions as  \eqref{full_system_2}.
Hence, given  $\tilde{\theta}_0\in L^2(\Omega)$ we study
\begin{equation}
\left\{
\begin{array}{rcll}
-\div \tilde{\ten{T}} &=& \vc{f} & \mbox{in } \Omega\times (0,T), \\
\tilde{\ten{T}} &=& \ten{D}\ten{\varepsilon}(\tilde{\vc{u}}) & \mbox{in } \Omega\times (0,T), \\
\tilde{\vc{u}} &=& \vc{g} & \mbox{on } \partial\Omega\times (0,T), 
\end{array}
\right.
\label{war_brz_u}
\end{equation}
and
\begin{equation}
\left\{
\begin{array}{rcll}
\tilde{\theta}_t -\Delta \tilde{\theta} &=& 0 & \mbox{in } \Omega\times (0,T), \\
\frac{\partial\tilde{\theta}}{\partial\vc{n}} &=& g_{\theta} & \mbox{on } \partial\Omega\times (0,T), \\
\tilde{\theta}(x,0) &=& \tilde{\theta}_0 & \mbox{in } \Omega.
\end{array}
\right.
\label{war_brz_t}
\end{equation}


\begin{lemat}
Let $\tilde{\theta}_0 \in L^2(\Omega)$, $\vc{g} \in L^p(0,T, W^{1-\frac{1}{p},p}(\partial\Omega,\mathbb{R}^3))$, $g_{\theta} \in L^2(0,T,L^2(\partial\Omega))$ and moreover  $\vc{f}\in L^p(0,T,W^{-1,p}(\Omega,\mathbb{R}^3))$. Then there exists a solution to systems \eqref{war_brz_u} and \eqref{war_brz_t}. Additionally, the following estimates hold:
\begin{equation}
\begin{split}
\|\tilde{\vc{u}}\|_{L^p(0,T,W^{1,p}(\Omega))} 
& \leq 
C_1 \left(\|\vc{g}\|_{L^p(0,T, W^{1-\frac{1}{p},p}(\partial\Omega,\mathbb{R}^3))}+ 
\|\vc{f}\|_{L^p(0,T,W^{-1,p}(\Omega))} \right),
 \\
\|\tilde{\theta}\|_{L^{\infty}(0,T,L^1(\Omega))}  + \|\tilde{\theta}\|_{L^2(0,T,W^{1,2}(\Omega))} 
& \leq 
C_2 \left(\|g_{\theta}\|_{L^2(0,T,L^2(\partial\Omega))}+\|\tilde{\theta}_0\|_{L^2(\Omega)} \right).
\nonumber
\end{split}
\end{equation}
\label{wyrzucenie_war_brzeg}
Moreover, $\theta$ belongs to $C([0,T],L^2(\Omega))$.
\end{lemat}
\begin{uwaga}
From the trace theorem \cite[Chapter II]{Valent} there exist $\tilde{\vc{g}}\in L^{p}(0,T,W^{1,p}(\Omega,\mathbb{R}^3))$ such that $\tilde{\vc{g}}|_{\partial\Omega}=\vc{g}$. Then, finding the solution $\tilde {\vc{u}}$ to \eqref{war_brz_u} is equivalent to finding the solution  $\tilde{\vc{u}}_1$ to the following problem 
%
\begin{equation}
\left\{
\begin{array}{rcll}
-{\div} \ten{D}\ten{\varepsilon}(\tilde{\vc{u}}_1) &=& \vc{f} +{\div} \ten{D}\ten{\varepsilon}(\vc{\tilde{g}}) & \mbox{in } \Omega\times (0,T), \\
\tilde{\vc{u}}_1 &=& 0 & \mbox{on } \partial\Omega\times (0,T),
\end{array}
\right.
\label{war_brz_u_0}
\end{equation}
and $\tilde{\vc{u}} = \tilde{\vc{u}}_1 + \tilde{\vc{g}}$. Using \cite[Corollary 4.4]{Valent}, we obtain the estimates presented in Lemma \ref{wyrzucenie_war_brzeg}.
\end{uwaga}




Instead of finding  $(\widehat{\vc u}, \widehat{\theta})-$ the solution to problem \eqref{full_system_2}-\eqref{init_0}-\eqref{boun_0}
we shall search for $(\vc{u}, \theta)$, where $\vc{u}=\widehat{\vc{u}}-\tilde{\vc{u}}$ and $\theta=\widehat{\theta}-\tilde{\theta}$ and $(\tilde{\vc{u}},\tilde{\theta})$ solve \eqref{war_brz_u} with ${\vc{g}}=0$ and \eqref{war_brz_t}. Furthermore, we get


\begin{equation}
\left\{
\begin{split}
- {\div} \ten{T} = - {\div} (\widehat{\ten{T}} - \tilde{\ten{T}}) & =  0 ,
\\
\ten{T} & =  \ten{D}(\ten{\varepsilon}(\vc{u}) - \ten{\varepsilon}^{\bf p} ),
\\
\ten{\varepsilon}^{\bf p}_t & =  \ten{G}(\widehat{\theta},\widehat{\ten{T}}^d) 
\\
& = \ten{G}(\theta + \tilde{\theta},\ten{T}^d+\tilde{\ten{T}}^d),
\\
\theta_t - \Delta \theta = (\widehat{\theta} - \tilde{\theta})_t - \Delta(\widehat{\theta} - \tilde{\theta}) & = \widehat{\ten{T}}^d:\ten{G}(\theta + \tilde{\theta},\ten{T}^d+\tilde{\ten{T}}^d) 
\\
& = \big(\ten{T}^d + \tilde{\ten{T}}^d\big):\ten{G}(\theta + \tilde{\theta},\ten{T}^d+\tilde{\ten{T}}^d).
\end{split}
\right.
\label{full_system_22a}
\end{equation}
Hence, we consider the problem
\begin{equation}
\left\{
\begin{split}
- {\div} \ten{T} & =  0 ,
\\
\ten{T} & = \ten{D}(\ten{\varepsilon}(\vc{u}) - \ten{\varepsilon}^{\bf p} ),
\\
\ten{\varepsilon}^{\bf p}_t & =  \ten{G}(\theta + \tilde{\theta}, \ten{T}^d + \tilde{\ten{T}}^d),
\\
\theta_t - \Delta \theta & =   \big(\ten{T}^d + \tilde{\ten{T}}^d\big):\ten{G}(\theta + \tilde{\theta}, \ten{T}^d + \tilde{\ten{T}}^d),
\end{split}
\right.
\label{full_system_22}
\end{equation}
with the initial and boundary conditions
\begin{equation}
\left\{	
\begin{array}{rcll}
\vc{u} &=& 0 & \mbox{on } \partial\Omega\times (0,T), \\
\frac{\partial\theta}{\partial \vc{n}} &=& 0 & \mbox{on } \partial\Omega\times (0,T), \\
\theta(\cdot,0) &=& \widehat{\theta}_0 - \tilde{\theta}_0 \equiv \theta_0 & \mbox{in } \Omega, \\
\ten{\varepsilon}^{\bf p}(\cdot,0) &=& \ten{\varepsilon}^{\bf p}_0 & \mbox{in } \Omega,
\end{array}
\right.
\label{in_bou_cond}
\end{equation}
where $\widehat{\theta}_0$ is the given initial condition for the temperature and 
$\tilde{\theta}_0$ is the initial condition for the system \eqref{war_brz_t}.




\begin{uwaga}
From the proof provided in Section \ref{sec:proof} it follows that the displacement $\ten{u}$, which is a solution to the homogeneous problem belongs to the space $C([0,T],L^{p'}(\Omega))$. However, in Theorem \ref{thm:main2}
the information on the continuity of the solution to the nonhomogeneous  problem does not appear. This is the consequence of the the fact that $\tilde{\ten{u}}$ may fail to be continuous under the assumptions that we have for the volume force $\ten{f}$ and boundary data.  
\end{uwaga}

\section{Proof of Theorem  \ref{thm:main2}}\label{sec:proof}
\subsection{Approximate solutions}
\label{sec:4}
Let $k\in{\mathbb N}$ and $\mathcal{T}_k(\cdot)$ be a standard truncation operator
\begin{equation}
\mathcal{T}_k(x)=\left\{
\begin{split}
k \qquad & x> k \\
x \qquad & |x|\leq k \\
-k \qquad & x <-k.
\end{split}
\right.
\label{Tk}
\end{equation}



We are facing the problem of low regularity  of the right hand side of the heat equation and the initial condition. Both functions are only integrable what enforces using  some delicate methods, such as the approach of  Boccardo and Gallou\"{e}t, cf. \cite{Boccardo}, for showing  the existence of solutions. An essential step is testing the equation with the truncation of solution. However, this truncation need not to be a linear combination of basis functions. This is the reason why  we use two level approximation, i.e. independent parameters of approximation in the displacement and temperature.
%
%
We pass to the limit, firstly with parameter $l$ corresponding to the dimension of the Galerkin basis for the temperature to get the sequence of infinite dimensional approximate solutions.  Passing to the limit with parameter $k$ corresponding to the dimension of the Galerkin basis for the displacement requires closer attention.

 We construct the approximated system using the Galerkin method. Consider the space $L^2(\Omega,\mathcal{S}^3)$ with a scalar product defined
 \begin{equation}
(\ten{\xi},\ten{\eta})_{\ten{D}}:=  \int_\Omega {\ten{D}}^\frac{1}{2}\ten{\xi}\cdot {\ten{D}}^\frac{1}{2}\ten{\eta} \dx 
\quad\mbox{for }\ten{\xi},\ten{\eta}\in L^2(\Omega,\mathcal{S}^3)
 \end{equation}
 where ${\ten{D}}^\frac{1}{2}\circ{\ten{D}}^\frac{1}{2}=\ten{D}$. 
Let $\{\vc{w}_i\}_{i=1}^{\infty}$ be the set of eigenfunctions of the operator $-\div\ten{D}\ten{\varepsilon}(\cdot)$ with the domain $W_0^{1,2}(\Omega,\mathbb{R}^3)$ and  $\{ \lambda_i \}$ be the corresponding eigenvalues   such that $\{\vc{w}_i\}$ is orthogonal in $W^{1,2}_0(\Omega,\mathbb{R}^3)$ with the inner product
\begin{equation}
( \vc{w}, \vc{v})_{W^{1,2}_0(\Omega)}=( \ten{\varepsilon}(\vc{w}), \ten{\varepsilon}(\vc{v}))_{\ten{D}}
\end{equation}
and orthonormal in $L^2(\Omega,\mathbb{R}^3)$. 
Hence
\begin{equation}
\|\ten{\varepsilon}(\vc{w})\|^2_{\ten{D}}=( \ten{\varepsilon}(\vc{w}), \ten{\varepsilon}(\vc{v}))_{\ten{D}}.
\end{equation}
Using the eigenvalue problem for the operator $-\div\ten{D}\ten{\varepsilon}(\cdot)$ we obtain
\begin{equation}
\int_{\Omega}\ten{D}\ten{\varepsilon}(\vc{w}_i):\ten{\varepsilon}(\vc{w}_j) \dx = \lambda_i \int_{\Omega}\vc{w}_i\cdot\vc{w}_j \dx = 0
\end{equation} 
Moreover, let $\{v_i\}_{i=1}^\infty$ be the set of  eigenfunctions of the  Laplace operator with the domain $W^{1,2}_n(\Omega)=\{ v\in W^{1,2}(\Omega):\quad \frac{\partial v}{\partial\vc{n}} = 0 \}$, let $\{\mu _i \}$ be the set of corresponding eigenvalues, let  $\{v_i\}$ be orthogonal in $W^{1,2}_n(\Omega)$ and orthonormal in $L^2(\Omega)$. These two families of vectors shall be used to construct  the finite dimensional approximations of the displacement and the temperature.  To construct the basis for approximating the visco-elastic strain tensor we will proceed as follows. 

 Let us consider the symmetric gradients of first 
$k$ functions from the basis $\{\ten{w}_i\}_{i=1}^{\infty}$. Due to the regularity of the eigenfunctions we observe that
$\ten{\varepsilon}(\ten{w}_i)$ are elements of $H^s(\Omega,\mathcal{S}^3)$, namely the fractional Sobolev space with a scalar product denoted by $\braket{\cdot,\cdot}_s$ and $s>\frac{3}{2}$.
Define now
\begin{equation}\label{Vk}
V_k:= (\mbox{span}\{\ten{\varepsilon}(\ten{w}_1),...,\ten{\varepsilon}(\ten{w}_k)\})^\bot,
\end{equation}
which is the orthogonal complement in $L^2(\Omega,\mathcal{S}^3)$  taken with respect to the scalar product $(\cdot,\cdot)_{\ten{D}}$ and also 
\begin{equation}\label{Vks}
V_k^s:=V_k\cap H^s(\Omega,\mathcal{S}^3)
\end{equation}
Let $\{\ten{\zeta}^k_n\}_{n=1}^{\infty}$ denote the orthonormal basis of  $V_k$,  which is also an  orthogonal basis of $V_k^s$, for more details see Appendix~\ref{B}.
For $k,l\in\mathbb{N}$, 
we are ready to define
\begin{equation}
\begin{split}
\vc{u}_{k,l} & = \sum_{n=1}^k\alpha_{k,l}^n(t) \vc{w}_n,
 \\
\theta_{k,l} & = \sum_{m=1}^l\beta_{k,l}^m(t) v_m,
 \\
\ten{\varepsilon}^{\bf p}_{k,l} & = \sum_{n=1}^k\gamma_{k,l}^n(t) \ten{\varepsilon}(\vc{w}_n) + 
\sum_{m=1}^l\delta_{k,l}^m(t) \ten{\zeta}_m^k,
\end{split}
\label{eq:postac}
\end{equation}

such that $\vc{u}_{k,l}$, $\ten{\varepsilon}^{\bf p}_{k,l}$ and $\theta_{k,l}$ solve the system of equations

\begin{equation}
\begin{array}{rll}
\int_{\Omega}  \ten{T}_{k,l} : \ten{\varepsilon}(\vc{w}_n) \dx &= 0
& n=1,...,k ,
\\[1ex]
\ten{T}_{k,l} &= \ten{D}(\ten{\varepsilon}(\vc{u}_{k,l}) - \ten{\varepsilon}^{\bf p}_{k,l} ),
\\[1ex]
\int_{\Omega}(\ten{\varepsilon}^{\bf p}_{k,l})_t : \ten{D}\ten{\varepsilon}(\vc{w}_n) \dx &= 
\int_{\Omega}\ten{G}(\theta_{k,l} + \tilde{\theta},  \ten{T}^d_{k,l} + \tilde{\ten{T}}^d  ) : \ten{D}\ten{\varepsilon}(\vc{w}_n) \dx 
& n=1,...,k ,
\\[1ex]
\int_{\Omega}(\ten{\varepsilon}^{\bf p}_{k,l})_t : \ten{D}\ten{\zeta}^k_m \dx &= 
\int_{\Omega}\ten{G}(\theta_{k,l} + \tilde{\theta},  \ten{T}^d_{k,l} + \tilde{\ten{T}}^d  ) : \ten{D}\ten{\zeta}^k_m \dx 
& m=1,...,l ,
\\[1ex]
\int_{\Omega}(\theta_{k,l})_t v_m\dx  + \int_{\Omega}\nabla\theta_{k,l}\cdot\nabla v_m \dx &
\\[1ex]
= \int_{\Omega} \mathcal{T}_k(  (\ten{T}_{k,l}^d + \tilde{\ten{T}}^d  ): & \ten{G}(\theta_{k,l} + \tilde{\theta},  \ten{T}^d_{k,l} + \tilde{\ten{T}}^d  ) ) v_m \dx & m=1,...,l .
\end{array}
\label{app_system}
\end{equation}
for a.a. $t\in(0,T)$.
For each approximate equation we  have the initial conditions in the following  form 
\begin{equation}
\left\{
\begin{array}{rclc}
\left( \theta_{k,l}(x,0), v_m\right) &=& \left( \mathcal{T}_k(\theta_0),v_m \right) & m=1,..,l, \\[1ex]
\left( \ten{\varepsilon}^{\bf p}_{k,l}(x,0), \ten{\varepsilon}(\vc{w}_n) \right)_{\ten{D}} &=& \left(\ten{\varepsilon}^{\bf p}_0, \ten{\varepsilon}(\vc{w}_n) \right)_{\ten{D}}
& n=1,..,k,
\\[1ex]
\left( \ten{\varepsilon}^{\bf p}_{k,l}(x,0), \ten{\zeta}_m^k) \right)_{\ten{D}} &=& \left(\ten{\varepsilon}^{\bf p}_0, \ten{\zeta}^k_m \right)_{\ten{D}}
& m=1,..,l,
\end{array}
\right.
\label{eq:warunki_pocz_app}
\end{equation}
where $\big(\cdot,\cdot\big)$ denotes the inner product in $L^2(\Omega)$ and $\big(\cdot,\cdot\big)_{\ten{D}}$ the inner product in $L^2(\Omega,\mathcal{S}^3)$.

Let us define 
\begin{equation}
\begin{split}
\vc{\xi}_1(t) &=(\alpha_{k,l}^1(t),..., \alpha_{k,l}^k(t))^T, \\
\vc{\xi}_2(t) &= (\beta_{k,l}^1(t),...,\beta_{k,l}^l(t),\gamma_{k,l}^1(t),..., \gamma_{k,l}^k(t),\delta_{k,l}^1(t),...,\delta_{k,l}^l(t) )^T .
\end{split}
\nonumber
\end{equation} 
The selection of the Galerkin bases and representation of the approximate solution \eqref{eq:postac} allows to notice that 
\begin{equation}
\alpha_{k,l}^n(t) =  \frac{1}{\lambda_n} 
\gamma_{k,l}^n(t)\int_{\Omega}\ten{D}\ten{\varepsilon}(\vc{w}_n):\ten{\varepsilon}(\vc{w}_n) \dx = \gamma_{k,l}^n(t)
\\
\end{equation}
and hence we obtain
\begin{equation}
\left\{
\begin{split}
(\gamma_{k,l}^n(t))_t  &= 
\frac{1}{\lambda_n} 
\int_{\Omega}\tilde{\ten{G}}(x,t,\vc{\xi}_1(t),\vc{\xi}_2(t)) : \ten{D}\ten{\varepsilon}(\vc{w}_n) \dx  ,
\\
(\delta_{k,l}^m(t))_t  &=
\int_{\Omega}\tilde{\ten{G}}(x,t,\vc{\xi}_1(t),\vc{\xi}_2(t)) : \ten{D}\ten{\zeta}_m^k\dx ,
\\
(\beta_{k,l}^m(t))_t &= \int_{\Omega} \mathcal{T}_k\Big(  \big(( \ten{D}\sum_{n=1}^k\alpha_{k,l}^n\ten{\varepsilon}(\vc{w}_n) - \ten{D}(\sum_{n=1}^l\gamma_{k,l}^n(t) \ten{\varepsilon}(\vc{w}_n) +
\delta_{k,l}^n(t) \ten{\zeta}_n ))^d + \tilde{\ten{T}}^d  \big)
\\
&\quad :\tilde{\ten{G}}(x,t,\vc{\xi}_1(t),\vc{\xi}_2(t)) \Big) v_m \dx  + \mu_m \beta_{k,l}^m(t),
\end{split}
\right.
\label{app_system20}
\end{equation}
for $n=1,...,k$ and $m=1,...,l$, where
\begin{equation}
\begin{split}
& \quad \tilde{\ten{G}}(x,t,\vc{\xi}_1(t),\vc{\xi}_2(t))
\\
& := \ten{G}(\theta_{k,l} +\tilde{\theta},\ten{T}_{k,l}^d + \tilde{\ten{T}}^d)
\\
&=\ten{G}\Big(\sum_{j=1}^l \beta_{k,l}^j(t) v_j(x) + \tilde{\theta},  \Big(\ten{D}\sum_{j=1}^k \alpha_{k,l}^j(t)\ten{\varepsilon}(\vc{w}_j) - \ten{D}\sum_{j=1}^l \big(\gamma_{k,l}^j(t) \ten{\varepsilon}(\vc{w}_j) + \delta_{k,l}^j(t) \ten{\zeta}_j  \big) \Big)^d + \tilde{\ten{T}}^d  \Big)
\end{split}
\nonumber
\end{equation}
Hence
\begin{equation}
\left\{
\begin{split}
(\gamma_{k,l}^n(t))_t  &= 
\frac{1}{\lambda_n} 
\int_{\Omega}\tilde{\ten{G}}(x,t,\vc{\xi}_1(t),\vc{\xi}_2(t)) : \ten{D}\ten{\varepsilon}(\vc{w}_n) \dx  ,
\\
(\delta_{k,l}^m(t))_t  &=
\int_{\Omega}\tilde{\ten{G}}(x,t,\vc{\xi}_1(t),\vc{\xi}_2(t)) : \ten{D}\ten{\zeta}_m^k\dx ,
\\
(\beta_{k,l}^m(t))_t &= \int_{\Omega} \mathcal{T}_k\Big(  \big(( \ten{D}\sum_{n=1}^k\alpha_{k,l}^n\ten{\varepsilon}(\vc{w}_n) - \ten{D}(\sum_{n=1}^l\gamma_{k,l}^n(t) \ten{\varepsilon}(\vc{w}_n) +
\delta_{k,l}^n(t) \ten{\zeta}_n ))^d + \tilde{\ten{T}}^d  \big)
\\
&\quad :\tilde{\ten{G}}(x,t,\vc{\xi}_1(t),\vc{\xi}_2(t)) \Big) v_m \dx  + \mu_m \beta_{k,l}^m(t),
\end{split}
\right.
\label{app_system2}
\end{equation}
System \eqref{app_system2} with initial conditions \eqref{eq:warunki_pocz_app} can be equivalently written  as the initial value problem 
\begin{equation}\label{47}
\begin{split}
&\frac{d\vc{\xi}_2 }{dt}  = \vc{F}(\vc{\xi}_1(t),\vc{\xi}_2(t),t),
\qquad
t\in [0,T),
\\
&\vc{\xi}_2(0) =\vc{\xi}_{2,0},
\end{split}
\end{equation}
where $\vc{\xi}_{2,0}$ is a vector of initial conditions obtained from \eqref{eq:warunki_pocz_app}. For $n\leq k$, we get $\alpha_{k,l}^n=\gamma_{k,l}^n$, hence $\vc{F}(\vc{\xi}_1(t),\vc{\xi}_2(t),t)$ can be treated as a function only of $\vc{\xi}_2(t)$, i.e. $\vc{F}(\vc{\xi}_1(t),\vc{\xi}_2(t),t)=\tilde{\vc{F}}(\vc{\xi}_2(t),t)$.
\begin{lemat}{(Existence of approximate solution)}

For initial condition satisfying $\ten{\varepsilon}^{\bf p}_0\in L^2(\Omega,\mathcal{S}^3_d)$ and $\theta_0\in L^1(\Omega)$ there exists an absolutely continuous in time solution to \eqref{47}.

\label{istnienie_przyblizone}
\end{lemat}


\begin{proof}

According to Carath\'eodory Theorem, see \cite[Theorem 3.4]{maleknecas} or \cite[Appendix $(61)$]{zeidlerB}, there exist  unique absolutely continuous functions $\beta_{k,l}^m(t)$, $\gamma_{k,l}^n(t)$ and $\delta_{k,l}^m(t)$ for every $n \leq k$ and $m \leq l$ on some time interval $[0,t^*]$. 
Moreover for every $n \leq k$ there exists a unique absolutely continuous function $\alpha_{k,l}^n(t)$.



\end{proof}


\subsection{Boundedness of approximate solutions}
\label{sec:5}

In this section we show the uniform boundedness of approximate solutions. As the considered model describes the physical phenomena, then it is obvious that the total energy should be finite. The total energy of the system consists of potential energy and thermal energy.

\begin{defi}
We say that $\mathcal{E}$
is the  potential energy if
\begin{equation}
\mathcal{E}(\ten{\varepsilon}(\vc{u}),\ten{\varepsilon}^{\bf p}): = \frac{1}{2}\int_{\Omega}\ten{D}(\ten{\varepsilon}(\vc{u}) - \ten{\varepsilon}^{\bf p}):(\ten{\varepsilon}(\vc{u}) - \ten{\varepsilon}^{\bf p} ) \dx .
\nonumber
\end{equation}
\label{energia}
\end{defi}

\begin{lemat}
There exists a constant $C$ which
is uniform with respect to  $k$ and $l$ such that
\begin{equation}
\sup_{t\in [0,T]} \mathcal{E}(\ten{\varepsilon}(\vc{u}_{k,l}) , \ten{\varepsilon}^{\bf p}_{k,l}) (t)
+ c \|\ten{T}_{k,l}^d + \tilde{\ten{T}}^d\|^p_{L^p(0,T,L^p(\Omega))}
\leq
C.
\end{equation}
\label{pom_2}
\end{lemat}

\begin{proof}
The potential energy is an absolutely continuous function and calculating the time derivative of  $\mathcal{E}(t)$ we get for a.a. $t\in[0,T]$
\begin{equation}
\begin{split}
\frac{d}{dt} \mathcal{E}(\ten{\varepsilon}(\vc{u}_{k,l}) , \ten{\varepsilon}^{\bf p}_{k,l}) 
& = 
\int_{\Omega}\ten{D}(\ten{\varepsilon}(\vc{u}_{k,l}) - \ten{\varepsilon}^{\bf p}_{k,l}):(\ten{\varepsilon}(\vc{u}_{k,l}))_t 
\dx 
\\
& \quad
-
\int_{\Omega}\ten{D}(\ten{\varepsilon}(\vc{u}_{k,l}) - \ten{\varepsilon}^{\bf p}_{k,l}): (\ten{\varepsilon}^{\bf p}_{k,l})_t \dx.
\end{split}
\label{pochodna}\end{equation}
In the first step we  multiply \eqref{app_system}$_{(1)}$ by $\{(\alpha_{k,l}^n)_t\}$ for each $n\leq k$.
Summing over $n=1,...,k$ 
we obtain
\begin{equation}
 \int_{\Omega}\ten{D}(\ten{\varepsilon}(\vc{u}_{k,l}) - \ten{\varepsilon}^{\bf p}_{k,l}): (\ten{\varepsilon}(\vc{u}_{k,l})_t \dx =
0.
\label{pierwsze_r}
\end{equation}
In the second step we  multiply   \eqref{app_system}$_{(4)}$ by  
 $\delta^m_{k,l}$ and 
 summing over 
 $m=1,...,l$, we obtain the identity, which is equivalent to
\begin{equation}
\int_{\Omega}(\ten{\varepsilon}^{\bf p}_{k,l})_t:\ten{T}_{k,l} \dx=
\int_{\Omega}\ten{G}(\tilde{\theta} + \theta_{k,l},\tilde{\ten{T}}^d + \ten{T}^d_{k,l}):\ten{T}_{k,l} \dx.
\label{drugie_r}
\end{equation}
Thus
\begin{equation}\label{ene}
\begin{split}
\frac{d}{dt} \mathcal{E}(\ten{\varepsilon}(\vc{u}_{k,l}) , \ten{\varepsilon}^{\bf p}_{k,l}) 
 = 
-
\int_{\Omega}\ten{G}(\tilde{\theta} + \theta_{k,l},\tilde{\ten{T}}^d + \ten{T}^d_{k,l}):\ten{T}^d_{k,l}\dx.
\end{split}
\end{equation}
Using Assumption 1c and the Young inequality we get
\begin{equation}
\begin{split}
\frac{d}{dt} \mathcal{E}(\ten{\varepsilon}(\vc{u}_{k,l}) , \ten{\varepsilon}^{\bf p}_{k,l}) 
& = 
- \int_{\Omega} (\ten{T}_{k,l}^d+\tilde{\ten{T}}^d) : \ten{G}(\theta_{k,l}+\tilde{\theta},\ten{T}_{k,l}^d +\tilde{\ten{T}}^d) \dx
\\
& \quad 
+ \int_{\Omega} \tilde{\ten{T}}^d : \ten{G}(\theta_{k,l}+\tilde{\theta},\ten{T}_{k,l}^d +\tilde{\ten{T}}^d) \dx    
\\
& 
\leq
- \beta \|\ten{T}_{k,l}^d + \tilde{\ten{T}}^d\|^p_{L^p(\Omega)} 
+ \|\tilde{\ten{T}}^d\|_{L^p(\Omega)}\|\ten{G}(\theta_{k,l}+\tilde{\theta},\ten{T}_{k,l}^d +\tilde{\ten{T}}^d)\|_{L^{p'}(\Omega)} 
\\
& \leq
- \beta \|\ten{T}_{k,l}^d + \tilde{\ten{T}}^d\|^p_{L^p(\Omega)} 
+ c(\epsilon)\|\tilde{\ten{T}}^d\|_{L^p(\Omega)}^p + \epsilon\|\ten{G}(\theta_{k,l}+\tilde{\theta},\ten{T}_{k,l}^d +\tilde{\ten{T}}^d)\|_{L^{p'}(\Omega)}^{p'} 
\nonumber
\end{split}\label{osz1}
\end{equation}
where  $\epsilon=\frac{\beta}{2^{p+1}C}$, with a constant $C$ coming from Assumption 1b.  Hence we estimate the last term as follows
\begin{equation}
\epsilon\|\ten{G}(\theta_{k,l}+\tilde{\theta},\ten{T}_{k,l}^d +\tilde{\ten{T}}^d)\|_{L^{p'}(\Omega)}^{p'} \le
\frac{\beta}{2}|\Omega|+\frac{\beta}{2} \|\ten{T}_{k,l}^d + \tilde{\ten{T}}^d\|^p_{L^p(\Omega)}.
\end{equation}
Finally, integrating over $(0,t)$, with $0\le t\le T$ we obtain
\begin{equation}\label{osz2}
\begin{split}
 \mathcal{E}(\ten{\varepsilon}(\vc{u}_{k,l}) , \ten{\varepsilon}^{\bf p}_{k,l}) (t)
&+ \frac{\beta}{2} \|\ten{T}_{k,l}^d + \tilde{\ten{T}}^d\|^p_{L^p(0,T,L^p(\Omega))}
\\&\leq
c(\epsilon)\|\tilde{\ten{T}}^d\|^p_{L^p(0,T,L^p(\Omega))}
+  \mathcal{E}(\ten{\varepsilon}(\vc{u}_{k,l}) , \ten{\varepsilon}^{\bf p}_{k,l})(0) +\frac{\beta}{2} |\Omega|.
\end{split}\end{equation}

\end{proof}

\begin{uwaga} 
From \eqref{osz2} we immediately observe that the sequence $\{\ten{T}_{k,l}^d\}$ is uniformly bounded in the space $L^p(0,T,L^p(\Omega,\mathcal{S}^3))$ with respect to $k$ and $l$.
Additionally, combining \eqref{osz1} and \eqref{osz2} we conclude the uniform boundedness of the sequence 
$\{\ten{G}(\theta_{k,l}+\tilde{\theta},\ten{T}_{k,l}^d +\tilde{\ten{T}}^d)\}$ in the space $L^{p'}(0,T,L^{p'}(\Omega,\mathcal{S}^3))$ and hence the uniform boundedness of the sequence $\{(\ten{T}_{k,l}^d +\tilde{\ten{T}}^d):\ten{G}(\theta_{k,l}+\tilde{\theta},\ten{T}_{k,l}^d +\tilde{\ten{T}}^d)\}$ in $L^1(0,T,L^1(\Omega))$.
\label{wsp_ogr_T}
\end{uwaga}

\begin{uwaga}
The uniform boundedness of the potential energy implies that the sequence $\{\ten{T}_{k,l}\}$ is uniformly bounded in $L^{\infty}(0,T,L^2(\Omega,\mathcal{S}^3))$ and in particular in $L^2(0,T,L^2(\Omega,\mathcal{S}^3))$.  
\end{uwaga}


\begin{lemat}
The sequence $\{(\ten{\varepsilon}^{\bf p}_{k,l})_t\}$ is uniformly bounded in $L^{p'}(0,T,(H^{s}(\Omega,\mathcal{S}^3))')$ with respect to $k$ and $l$. 
\label{wsp_org_epa}
\end{lemat}

\begin{proof}
Let $P^l$ be a projection on ${\rm lin}\{\ten{\zeta}_1,\ldots,\ten{\zeta}_l\}$, $P^l(\ten{v}):=\sum_{i=1}^{l}(\ten{v},\ten{\zeta}_i)_{\ten{D}}\ten{\zeta}_i$, then $\|P^l\varphi\|_{H^s}\le\|\varphi\|_{H^s}$. 
Let $P^k$ be a projection on ${\rm lin}\{\ten{\varepsilon}(\vc{w}_1),\ldots,\ten{\varepsilon}(\vc{w}_1)\}$, 
$P^k(\ten{v}):=\sum_{i=1}^{k}(\ten{v},\ten{\varepsilon}(\vc{w}_i))_{\ten{D}}\ten{\varepsilon}(\vc{w}_i)$.
Since $P^k$ is the projection of a finite dimensional space, and the dimension of the space is independent of $l$, there exists a constant, also independent of $l$ such that 
 $\|P^k\varphi\|_{H^s}\le c\|\varphi\|_{H^s}$ 
 Let $\varphi\in L^p(0,T,H^{s}(\Omega,\mathcal{S}^3))$ and we may estimate as follows
\begin{equation}
\begin{split}
\int_0^T |\langle (\ten{\varepsilon}^{\bf p}_{k,l})_t, \varphi\rangle |\dt &=
\int_0^T |\langle (\ten{\varepsilon}^{\bf p}_{k,l})_t, (P^k + P^l)\varphi\rangle |\dt
\\ &
\le\int_0^T |\langle (\ten{\varepsilon}^{\bf p}_{k,l})_t, P^k\varphi\rangle |\dt
+\int_0^T |\langle (\ten{\varepsilon}^{\bf p}_{k,l})_t, P^l\varphi\rangle |\dt ,
\end{split}
\end{equation}
where the equality results from orthogonality of subspaces $\mbox{lin}\{\ten{\varepsilon}(\vc{w}_1),\ldots,\ten{\varepsilon}(\vc{w}_k)\}$ and $\mbox{lin}\{\ten{\zeta}_1,\ldots, \ten{\zeta}_l\}$. Then
\begin{equation}
\begin{split}
\int_0^T |\langle (\ten{\varepsilon}^{\bf p}_{k,l})_t, \varphi\rangle |\dt &\le
\int_0^T |\int_\Omega
\ten{G}(\theta_{k,l}+\tilde{\theta},\ten{T}_{k,l}^d +\tilde{\ten{T}}^d) P^k\varphi\dx |\dt
\\ &
\quad +
\int_0^T |\int_\Omega
\ten{G}(\theta_{k,l}+\tilde{\theta},\ten{T}_{k,l}^d +\tilde{\ten{T}}^d) P^l\varphi\dx |\dt
\\ &\le 
\int_0^T\|\ten{G}(\theta_{k,l}+\tilde{\theta},\ten{T}_{k,l}^d +\tilde{\ten{T}}^d)\|_{L^{p'}(\Omega)}
\|P^k\varphi\|_{L^{p}(\Omega)}\dt\\
&
\quad + \int_0^T\|\ten{G}(\theta_{k,l}+\tilde{\theta},\ten{T}_{k,l}^d +\tilde{\ten{T}}^d)\|_{L^{p'}(\Omega)}
\|P^{l}\varphi\|_{L^{p}(\Omega)}\dt
\\ &
\le  \tilde c\int_0^T\|\ten{G}(\theta_{k,l}+\tilde{\theta},\ten{T}_{k,l}^d +\tilde{\ten{T}}^d)\|_{L^{p'}(\Omega)}
\|P^k\varphi\|_{H^{s}(\Omega)}\dt
\\ &
\quad + \tilde c \int_0^T\|\ten{G}(\theta_{k,l}+\tilde{\theta},\ten{T}_{k,l}^d +\tilde{\ten{T}}^d)\|_{L^{p'}(\Omega)}
\|P^{l}\varphi\|_{H^{s}(\Omega)}\dt
\\ &
\le c\tilde c\int_0^T\|\ten{G}(\theta_{k,l}+\tilde{\theta},\ten{T}_{k,l}^d +\tilde{\ten{T}}^d)\|_{L^{p'}(\Omega)}
\|\varphi\|_{H^{s}(\Omega)}\dt
\\ &
\quad + \tilde c\int_0^T\|\ten{G}(\theta_{k,l}+\tilde{\theta},\ten{T}_{k,l}^d +\tilde{\ten{T}}^d)\|_{L^{p'}(\Omega)}
\|\varphi\|_{H^{s}(\Omega)}\dt\\
&\le (1+c)\tilde c\|\ten{G}(\theta_{k,l}+\tilde{\theta},\ten{T}_{k,l}^d +\tilde{\ten{T}}^d)\|_{L^{p'}(0,T,L^{p'}(\Omega))}
\|\varphi\|_{L^p(0,T,H^{s}(\Omega))},\\
\end{split}\end{equation}
where  $\tilde c$ is an optimal embedding  constant of $H^s(\Omega,\mathcal{S}^3)\subset L^2(\Omega,\mathcal{S}^3)$. 
Consequently, there exists $C>0$ such that 
\begin{equation}
\sup_{\varphi\in L^p(0,T,H^{s}(\Omega))\atop
 \|\varphi\|_{L^p(0,T,H^{s}(\Omega))}\le 1 }\int_0^T
|\langle (\ten{\varepsilon}^{\bf p}_{k,l})_t, \varphi\rangle |\dt\le C
\end{equation}
and hence sequence $\{(\ten{\varepsilon}^{\bf p}_{k,l})_t\}$ is uniformly bounded in 
${L^{p'}(0,T,(H^{s}(\Omega,\mathcal{S}^3))')}$
\end{proof}

\begin{lemat}\label{LinftyL1}
The sequence $\{\theta_{k,l}\}$ is uniformly bounded in $L^\infty(0,T;L^1(\Omega))$ with respect to $k$ and $l$.
\end{lemat}
Since it can be immediately observed that 
\begin{equation}
\sup_{0\leq t\leq T}\|\theta_{k,l}(t)\|_{L^1(\Omega)}\leq
C(1+\|\ten{T}_{k,l}^d + \tilde{\ten{T}}^d\|_{L^p(0,T,L^p(\Omega))})+
\|\theta_0\|_{L^1(\Omega)}
\nonumber
\end{equation}
and Lemma~\ref{pom_2} holds, we omit the details of the proof. 
 The lemma provides that  the internal energy of $\Omega$ is finite at any time $t\in [0,T]$. It is possible to prove  better estimates for the temperature, however they are uniform only with respect to $l$ and not with respect to $k$. We provide the details in the proceeding lemma.

\begin{lemat}
There exists a constant $C$, depending on the domain $\Omega$ and the time interval $(0,T)$, such that for every $k\in\mathbb{N}$
\begin{equation}
\begin{split}
\sup_{0\leq t\leq T}&\|\theta_{k,l}(t)\|^2_{L^2(\Omega)} +
\|\theta_{k,l}\|^2_{L^2(0,T,W^{1,2}(\Omega))} +
\|(\theta_{k,l})_t\|^2_{L^2(0,T,W^{-1,2}(\Omega))}
\\
& \leq C\Big(\|\mathcal{T}_k\Big( (\ten{T}_{k,l}^d + \tilde{\ten{T}}^d  ):\ten{G}(\theta_{k,l}+\tilde{\theta},\ten{T}_{k,l}^d+\tilde{\ten{T}}^d) \Big)\|^2_{L^2(0,T,L^2(\Omega))}+
\|\mathcal{T}_k(\theta_0)\|_{L^2(\Omega)}^2\Big).
\label{numer}
\end{split}
\end{equation}
\label{lm:7}
\end{lemat}

The proof follows from  the standard tools for  parabolic equations, see e.g. Evans \cite{Evans}.

\begin{uwaga}
The uniform boundedness of solutions (Lemma \ref{pom_2} and Lemma \ref{lm:7}) implies the global existence of approximate solutions, i.e. existence of solutions $\{\beta_{k,l}^m(t),\gamma_{k,l}^n(t),\delta_{k,l}^m(t)\}$ on the whole time interval $[0,T]$ for each $n=1,...,k$ and $m=1,...,l$. Moreover, there exist global solutions $\{\alpha_{k,l}^n(t)\}$ for all $n=1,...,k$.
\end{uwaga}

\subsection{Limit passage $l\to\infty$ and uniform estimates.
}
\label{sec:7}
Before we pass to the limit let us multiply the system \eqref{app_system} by  smooth time-dependent functions, integrate over $[0,T]$ and then 
rewrite the system as follows
\begin{equation}\label{58}
\begin{split}
\int_0^T\int_{\Omega}\ten{T}_{k,l}:\nabla\vc{w}_n  \varphi_1(t) \dxdt 
&= 0, \quad n=1,\dots, k 
\end{split}
\end{equation}
\begin{equation}\label{59}\begin{split}
\int_0^T\int_{\Omega}(\ten{\varepsilon}^{\bf p}_{k,l})_t &: \ten{D}\ten{\varepsilon}(\vc{w}_n) \varphi_2(t)\dxdt \\&= 
\int_0^T\int_{\Omega}\ten{G}(\theta_{k,l} + \tilde{\theta},  \ten{T}^d_{k,l} + \tilde{\ten{T}}^d  ) : \ten{D}\ten{\varepsilon}(\vc{w}_n) 
\varphi_2(t)\dxdt, 
\quad n=1,...,k ,
\\
\int_0^T\int_{\Omega}(\ten{\varepsilon}^{\bf p}_{k,l})_t &: \ten{D}\ten{\zeta}^k_m \varphi_3(t)\dxdt \\&= 
\int_0^T\int_{\Omega}\ten{G}(\theta_{k,l} + \tilde{\theta},  \ten{T}^d_{k,l} + \tilde{\ten{T}}^d  ) : \ten{D}\ten{\zeta}^k_m \varphi_3(t)\dxdt, 
\quad  m=1,...,l ,
\end{split}\end{equation}
and for $m=1,\ldots,l$
\begin{equation}\label{60}
\begin{split}
&-\int_0^T\int_{\Omega} \theta_{k,l}\varphi_4'(t) v_m \dxdt -
\int_{\Omega} \theta_0(x)\varphi_4(0) v_m \dx   + 
\int_0^T\int_{\Omega}  \nabla\theta_{k,l}\cdot \varphi_4(t)\nabla v_m  \dxdt\\ & = 
\int_0^T\int_{\Omega} \mathcal{T}_k\left((\ten{T}^d_{k,l}+\tilde{\ten{T}}^d):\ten{G}(\theta_{k,l}+\tilde{\theta},\ten{T}^d_{k,l}+\tilde{\ten{T}}^d)\right)\varphi_4(t)v_m \dxdt,
\end{split}
\end{equation}
holds for every test functions ${\varphi}_1, {\varphi}_2, {\varphi}_3\in C^{\infty}([0,T])$ and $\varphi_4\in C^{\infty}_c([-\infty,T))$. 

Firstly, we pass to the limit with $l \rightarrow \infty$ - the Galerkin approximation of temperature. From the previous section we get uniform boundedness with respect to $l$ for appropriate sequences. Then at least for a subsequence, but still denoted by the index $l$, we get the following convergences

\begin{equation}
\begin{array}{cl}
\ten{T}_{k,l}\rightharpoonup \ten{T}_k  &  \mbox{weakly in }   L^2(0,T,L^2(\Omega,\mathcal{S}^3)),\\
\ten{T}^d_{k,l}\rightharpoonup \ten{T}^d_k  &  \mbox{weakly in }   L^p(0,T,L^p(\Omega,\mathcal{S}^3_d)),\\
\ten{G}(\tilde{\theta} + \theta_{k,l},\tilde{\ten{T}}^d + \ten{T}_{k,l}^d)\rightharpoonup \ten{\chi}_k  & \mbox{weakly in } L^{p'}(0,T,L^{p'}(\Omega,\mathcal{S}^3_d)), \\
\theta_{k,l}\rightharpoonup \theta_k  &  \mbox{weakly in }   L^2(0,T,W^{1,2}(\Omega)),\\
\theta_{k,l}\rightarrow \theta_k  &  \mbox{a.e. in } \Omega \times (0,T),\\
(\ten{\varepsilon}^{\bf p}_{k,l})_t\rightharpoonup(\ten{\varepsilon}^{\bf p}_{k})_t&  \mbox{weakly in }   
L^{p'}(0,T,(H^s(\Omega,\mathcal{S}^3))').
\end{array}
\end{equation}
Passing now to the limit in \eqref{58}-\eqref{59} yields
\begin{equation}\label{limit1}
\begin{split}
\int_0^T\int_{\Omega}\ten{T}_{k}:\nabla\vc{w}_n  \varphi_1(t) \dxdt 
&= 0, \quad n=1,\dots, k 
\end{split}
\end{equation}
\begin{equation}\label{limit2}\begin{split}
\int_0^T\int_{\Omega}(\ten{\varepsilon}^{\bf p}_{k})_t : \ten{D}\ten{\varepsilon}(\vc{w}_n) \varphi_2(t)\dxdt = 
\int_0^T\int_{\Omega}\ten{\chi}_{k}: \ten{D}\ten{\varepsilon}(\vc{w}_n) 
\varphi_2(t)\dxdt, 
\quad n=1,...,k ,
\\
\int_0^T\int_{\Omega}(\ten{\varepsilon}^{\bf p}_{k})_t : \ten{D}\ten{\zeta}^k_m \varphi_3(t)\dxdt = 
\int_0^T\int_{\Omega}\ten{\chi}_{k}  : \ten{D}\ten{\zeta}^k_m \varphi_3(t)\dxdt, 
\quad  m\in\mathbb{N},
\end{split}\end{equation}
holds for every test functions ${\varphi}_1, {\varphi}_2, {\varphi}_3\in C^{\infty}([0,T])$.
By the density of $\mbox{lin}\{\ten{\zeta}^k_m\}_{m=1}^\infty$ in $L^{p}(\Omega,\mathcal{S}^3)$ we conclude that 
\begin{equation}\label{65}
\int_0^T\int_{\Omega}(\ten{\varepsilon}^{\bf p}_{k})_t : \ten{\varphi}\dxdt = 
\int_0^T\int_{\Omega}\ten{\chi}_{k}  :  \ten{\varphi}\dxdt
 \end{equation}
holds for all $ \ten{\varphi}\in C^\infty([0,T],L^{p}(\Omega,\mathcal{S}^3))$ and then also for all 
$ \ten{\varphi}\in L^{p}(0,T;L^{p}(\Omega,\mathcal{S}^3))$.
In the rest of this section we identify the weak limit of the nonlinear term $\ten{\chi}_k$ and then show the convergence of  
$$\int_0^T\int_\Omega\mathcal{T}_k\left(\ten{G}(\tilde{\theta} + \theta_{k,l},\tilde{\ten{T}}^d + \ten{T}_{k,l}^d):(\tilde{\ten{T}}^d+\ten{T}^d_{k,l})\right)\dxdt$$ what shall allow to pass to the limit in \eqref{60}.


\begin{lemat}
The sequence $\{\ten{\varepsilon}^{\bf p}_{k}\}$ is uniformly bounded in $W^{1,p'}(0,T,L^{p'}(\Omega,\mathcal{S}^3))$ with respect to $k$. 
\label{wsp_org_ep}
\end{lemat}

\begin{proof}
By Assumption 1b 
and the fact that the constant $C$ is independent of temperature, we get

\begin{equation}
\ten{\varepsilon}^{\bf p}_{k}(x,t) = \ten{\varepsilon}^{\bf p}_{k}(x,0) + \int_0^t(\ten{\varepsilon}^{\bf p}_{k}(x,s))_s \ds .
\nonumber
\end{equation}
Hence
\begin{equation}
|\ten{\varepsilon}^{\bf p}_{k}|^{p'}(x,t) \leq c|\ten{\varepsilon}^{\bf p}_{k}|^{p'}(x,0) + ct^{1/p} \int_0^t|(\ten{\varepsilon}^{\bf p}_{k})_s|^{p'}(x,s) \ds
\nonumber
\end{equation}
and consequently
\begin{equation}
\begin{split}
\int_0^T\int_{\Omega}|\ten{\varepsilon}^{\bf p}_{k}|^{p'}(x,t) \dxdt &\leq
c\int_0^T\int_{\Omega}|\ten{\varepsilon}^{\bf p}_{k}|^{p'}(x,0) \dxdt  + ct^{1/p} \int_0^T\int_{\Omega}\int_0^t|(\ten{\varepsilon}^{\bf p}_{k})_s|^{p'}(x,s)\ds \dxdt
\\
&\leq C(T) (1 + \int_{\Omega}\int_0^T|\ten{G}(\theta_{k}+\tilde{\theta},\ten{T}_{k}^d +\tilde{\ten{T}}^d )|^{p'}) \ds \dx
 \\
&\leq C(T)(1 + \int_0^T\int_{\Omega}|\ten{T}_{k}^d+\tilde{\ten{T}}^d|^p) \dxdt .
\nonumber
\end{split}
\end{equation}
It follows from the previous lemma that the right hand side is uniformly bounded.
\end{proof}
\begin{lemat}
The sequence $\{\vc{u}_{k}\}$ is uniformly bounded in $L^{p'}(0,T,W^{1,p'}_0(\Omega,\mathbb{R}^3))$ with respect to $k$.
\label{wsp_org_u}
\end{lemat}
\begin{proof}
In view of Lemma \ref{pom_2} the sequence $\{\ten{T}_{k}\}$ is uniformly bounded in $L^2(0,T,L^2(\Omega,\mathcal{S}^3))$. Using the triangle inequality and boundedness of the operator $\ten{D}$ we obtain
\begin{equation}
|\ten{\varepsilon}(\vc{u}_{k})|^{p'}\leq c|\ten{\varepsilon}(\vc{u}_{k}) - \ten{\varepsilon}^{\bf p}_{k}|^{p'} + c|\ten{\varepsilon}^{\bf p}_{k}|^{p'}
\leq c|\ten{T}_{k}|^{p'} + c|\ten{\varepsilon}^{\bf p}_{k}|^{p'}.
\end{equation}
Integrating over $\Omega\times (0,T)$ and using that $1<p'\leq 2$ we get
\begin{equation}
\begin{split}
\int_0^T\int_{\Omega} |\ten{\varepsilon}(\vc{u}_{k})|^{p'} \dxdt
& \leq c \int_0^T\int_{\Omega} |\ten{T}_{k}|^{p'} \dxdt
+ c\int_0^T\int_{\Omega} |\ten{\varepsilon}^{\bf p}_{k}|^{p'} \dxdt
\\
& \leq c \int_0^T\int_{\Omega} |\ten{T}_{k}|^{2} \dxdt
+ c\int_0^T\int_{\Omega} |\ten{\varepsilon}^{\bf p}_{k}|^{p'} \dxdt
\\
& \leq c \|\ten{T}_{k}\|^2_{L^2(0,T,L^2(\Omega))}
+ c\|\ten{\varepsilon}^{\bf p}_{k}\|^{p'}_{L^{p'}(0,T,L^{p'}(\Omega))}  .
\end{split}
\end{equation}
The tensor $\ten{\varepsilon}(\vc{u}_{k})$ is the symmetric gradient of the displacement, thus using the  Korn inequality 
(cf.~\cite[Theorem 1.10]{maleknecas}) we conclude that the sequence $\{\vc{u}_{k}\}$ is uniformly bounded in $L^{p'}(0,T,W^{1,p'}_0(\Omega,\mathbb{R}^3))$. 
\end{proof}

\begin{lemat}
The following inequality holds for the solution of approximate system
\begin{equation}
\limsup_{l\rightarrow\infty}\int_{0}^{t}\int_{\Omega}\ten{G}(\tilde{\theta} + \theta_{k,l},\tilde{\ten{T}}^d + \ten{T}^d_{k,l}):\ten{T}^d_{k,l} \dxdt \leq
\int_{0}^{t}\int_{\Omega}\ten{\chi}_k:\ten{T}^d_k \dxdt .
\label{teza-8}
\end{equation}
\label{lm:8}
\end{lemat}

\begin{proof}
For each $\mu>0, t_2\le T-\mu, s\ge0, $ let $\psi_\mu:\mathbb{R}_+\to\mathbb{R}_+$ be defined as follows
\begin{equation}\label{psi-mu}
\psi_{\mu,t_2}(s)=\left\{
\begin{array}{lcl}
1&{\rm for}&s\in[0,t_2),\\
-\frac{1}{\mu}s+\frac{1}{\mu}t_2+1&{\rm for}&s\in[t_2, t_2+\mu),\\
0&{\rm for}&s\ge t_2+\mu.
\end{array}\right.
\end{equation}

Next we shall use \eqref{ene}  and  multiply it  by $\psi_{\mu,t_2}(t)$ and integrate over $(0,T)$
\begin{equation}\label{mu2}
\begin{split}
\int_{0}^{T}
\frac{d}{d\tau} \mathcal{E}(\ten{\varepsilon}(\vc{u}_{k,l}) , \ten{\varepsilon}^{\bf p}_{k,l}) \,\psi_{\mu,t_2}\dt
\
 = 
-
\int_0^T\int_{\Omega}\ten{G}(\tilde{\theta} + \theta_{k,l},\tilde{\ten{T}}^d + \ten{T}^d_{k,l}):\ten{T}^d_{k,l} \,\psi_{\mu,t_2}\dxdt.
\end{split}
\end{equation}
 Let us now integrate by parts the left hand side of \eqref{mu2}
\begin{equation}\label{mu3}
\begin{split}
\int_{0}^{T}
\frac{d}{d\tau} \mathcal{E}(\ten{\varepsilon}(\vc{u}_{k,l}) , \ten{\varepsilon}^{\bf p}_{k,l}) \,\psi_{\mu,t_2}\dt
=\frac{1}{\mu}\int_{t_2}^{t_2+\mu}\mathcal{E}(\ten{\varepsilon}(\vc{u}_{k,l}(t)) , \ten{\varepsilon}^{\bf p}_{k,l}(t)) \dt-
\mathcal{E}(\ten{\varepsilon}(\vc{u}_{k,l}(0)) , \ten{\varepsilon}^{\bf p}_{k,l}(0)).
\end{split}\end{equation}
Passing to the limit in \eqref{mu3} with $l\to\infty$ we obtain
\begin{equation}\label{mu4}
\begin{split}
\liminf\limits_{l\to\infty}\int_{0}^{T}
\frac{d}{d\tau}& \mathcal{E}(\ten{\varepsilon}(\vc{u}_{k,l}) , \ten{\varepsilon}^{\bf p}_{k,l}) \,\psi_{\mu,t_2}\dt\\
&=\liminf\limits_{l\to\infty}\frac{1}{\mu}\int_{t_2}^{t_2+\mu}\mathcal{E}(\ten{\varepsilon}(\vc{u}_{k,l}) , \ten{\varepsilon}^{\bf p}_{k,l}) \dt-
\lim\limits_{l\to\infty}\mathcal{E}(\ten{\varepsilon}(\vc{u}_{k,l}(0)) , \ten{\varepsilon}^{\bf p}_{k,l}(0))\\
&\ge \frac{1}{\mu}\int_{t_2}^{t_2+\mu}\mathcal{E}(\ten{\varepsilon}(\vc{u}_{k}(t)) , \ten{\varepsilon}^{\bf p}_{k}(t)) \dt-
\mathcal{E}(\ten{\varepsilon}(\vc{u}_{k}(0)) , \ten{\varepsilon}^{\bf p}_{k}(0))
\end{split}\end{equation}

Note that the last inequality holds due to the weak lower semicontinuity in 
$L^2(0,T,L^2(\Omega;\mathcal{S}^3))$.
To complete the proof  we choose in \eqref{limit1} the test functions
 $\varphi_1(t)=((\alpha^n_k)_t*\eta_{\epsilon}\mathbf{1}_{(t_1,t_2)})*\eta_{\epsilon}$,  
and  in \eqref{65} $\ten{\varphi}=(\ten{T}_k^d*\eta_{\epsilon}\mathbf{1}_{(t_1,t_2)})*\eta_{\epsilon}$,
where $\eta_\epsilon$ is a standard mollifier and we mollify with respect to time. Thus we  obtain
\begin{equation}
\begin{split}
\int_0^T\int_{\Omega}  \ten{T}_{k} : \ten{\varepsilon}(((\alpha^n_k)_t*\eta_{\epsilon}\mathbf{1}_{(t_1,t_2)})*\eta_{\epsilon}\vc{w}_n) \dx &= 0,
\\
\int_0^T\int_{\Omega}(\ten{\varepsilon}^{\bf p}_{k})_t : (\ten{T}_k^d*\eta_{\epsilon}\mathbf{1}_{(t_1,t_2)})*\eta_{\epsilon} \dx 
= 
\int_0^T\int_{\Omega}\ten{\chi}_{k} : &(\ten{T}_k^d*\eta_{\epsilon}\mathbf{1}_{(t_1,t_2)})*\eta_{\epsilon}\dx ,
\end{split}
\label{app_system_n}
\end{equation}
for  $n=1,...,k $. Summing \eqref{app_system_n}$_{(1)}$ over $n=1,...,k$ we obtain
\begin{equation}
\int_{t_1}^{t_2} \int_{\Omega}\ten{D}\left(\ten{\varepsilon}(\vc{u}_{k}) - \ten{\varepsilon}^{\bf p}_{k}\right)*\eta_{\epsilon}: (\ten{\varepsilon}(\vc{u}_{k})*\eta_{\epsilon})_t \dxdt =
0.
\label{pierwsze_r1}
\end{equation}
and
\begin{equation}
\int_{t_1}^{t_2}\int_{\Omega}(\ten{\varepsilon}^{\bf p}_{k}*\eta_{\epsilon})_t:\ten{T}_{k}*\eta_{\epsilon} \dxdt =
\int_{t_1}^{t_2}\int_{\Omega}\ten{\chi}_{k}*\eta_{\epsilon}:\ten{T}_{k}*\eta_{\epsilon} \dxdt .
\label{drugie_r2}
\end{equation}
Products in \eqref{drugie_r2} are well defined, since for the matrices $\ten{A}\in\mathcal{S}^3_d$ and $\ten{B}\in\mathcal{S}^3$ the equivalence $\ten{A}:\ten{B}^d=\ten{A}:\ten{B}$ holds and the sequence $\{\ten{T}^d_{k}\}$ is uniformly bounded in $L^{p'}(0,T,L^{p'}(\Omega,\mathcal{S}^3_d))$.
\noindent
Passing with $\epsilon\to0$ 
we obtain the equality
\begin{equation}
\frac{1}{2}\int_{\Omega}\ten{D}(\ten{\varepsilon}(\vc{u}_k) - \ten{\varepsilon}^{\bf p}_k):(\ten{\varepsilon}(\vc{u}_k) - \ten{\varepsilon}^{\bf p}_k) \dx \Big|_{t_1}^{t_2}
 = - 
\int_{t_1}^{t_2}\int_{\Omega}\ten{\chi}_k:\ten{T}^d_k \dxdt .
\label{granica_l}
\end{equation}
Since $\ten\varepsilon({\vc{u}_k}),  \ten{\varepsilon}^{\bf p}_k\in C_{w}([0,T],L^2(\Omega,\mathcal{S}^3))$, then we may pass
with $t_1\to 0$ and conclude
\begin{equation}\label{gran}
\mathcal{E}(\ten{\varepsilon}(\vc{u}_{k}(t_2)) , \ten{\varepsilon}^{\bf p}_{k}(t_2)) -
\mathcal{E}(\ten{\varepsilon}(\vc{u}_{k}(0)) , \ten{\varepsilon}^{\bf p}_{k}(0))=- 
\int_{0}^{t_2}\int_{\Omega}\ten{\chi}_k:\ten{T}^d_k \dxdt .
\end{equation}
Multiplying \eqref{gran} by  $\frac{1}{\mu}$ and integrating     over the interval $(t_2, t_2+\mu)$ we get
\begin{equation}
\begin{split}
\frac{1}{\mu}\int_{t_2}^{t_2+\mu}\mathcal{E}(\ten{\varepsilon}(\vc{u}_{k}(t)) , \ten{\varepsilon}^{\bf p}_{k}(t)) \dt-
\mathcal{E}(\ten{\varepsilon}(\vc{u}_{k}(0)) , \ten{\varepsilon}^{\bf p}_{k}(0))=-
\frac{1}{\mu}\int_{t_2}^{t_2+\mu}\int_{0}^{\tau}\int_{\Omega}\ten{\chi}_k:\ten{T}^d_k \dxdt \dtau.
\end{split}\end{equation}
For brevity we denote $$F(s):=\int_{\Omega}\ten{\chi}_k:\ten{T}^d_k\dx$$ which is obviously in $L^1(0,T)$. Then we may apply the Fubini theorem
\begin{equation}\label{71}\begin{split}
\frac{1}{\mu}\int_{t_2}^{t_2+\mu}\int_0^\tau F(s) \ds\dtau&=\frac{1}{\mu}\int_{\mathbb{R}^2} 
\mathbf{1}_{\{0\le s\le \tau\}}(s)\mathbf{1}_{\{t_2\le \tau\le t_2+\mu\}}(\tau) F(s)\ds\dtau\\
&=\frac{1}{\mu}\int_\mathbb{R} \left(\int_\mathbb{R}
\mathbf{1}_{\{0\le s\le\tau\}}(s)\mathbf{1}_{\{t_2\le \tau\le t_2+\mu\}} (\tau)\dtau\right)F(s)\ds.
\end{split}
\end{equation}
The crucial  observation is that
\begin{equation}\label{psi}
\psi_{\mu,t_2}(s)=\frac{1}{\mu}\int_\mathbb{R}
\mathbf{1}_{\{0\le t\le\tau\}}(t)\mathbf{1}_{\{t_2\le \tau\le t_2+\mu\}} (\tau)\dtau.
\end{equation}
Hence  using \eqref{mu2} and \eqref{mu4} we conclude
\begin{equation}
-\int_{0}^{T}\int_{\Omega}\ten{\chi}_k:\ten{T}^d_k \, \psi_{\mu,t_2}\dxdt \le
\liminf\limits_{l\to\infty}\left( 
-
\int_0^T\int_{\Omega}\ten{G}(\tilde{\theta} + \theta_{k,l},\tilde{\ten{T}}^d + \ten{T}^d_{k,l}):\ten{T}^d_{k,l} \,\psi_{\mu,t_2}\dxdt \right)
\end{equation}
which is nothing else than 
\begin{equation}
\limsup_{l\rightarrow\infty}\int_{0}^{T}\int_{\Omega}\ten{G}(\tilde{\theta} + \theta_{k,l},\tilde{\ten{T}}^d + \ten{T}^d_{k,l}):\ten{T}^d_{k,l} \ \psi_{\mu,t_2}\dxdt \leq
\int_{0}^{T}\int_{\Omega}\ten{\chi}_k:\ten{T}^d_k \ \psi_{\mu,t_2}\dxdt .
\label{jedna_nierownosc}
\end{equation}
Observe now that
\begin{equation}
\begin{split}
\limsup\limits_{l\to\infty}&\int_{0}^{t_2}\int_{\Omega}\ten{G}(\tilde{\theta} + \theta_{k,l},\tilde{\ten{T}}^d + \ten{T}^d_{k,l}):\ten{T}^d_{k,l} \dxdt\\
&\le
\limsup\limits_{l\to\infty}\int_{0}^{t_2}\int_{\Omega}\ten{G}(\tilde{\theta} + \theta_{k,l},\tilde{\ten{T}}^d + \ten{T}^d_{k,l}):(\tilde{\ten{T}}^d + \ten{T}^d_{k,l}) \dxdt\\
&-\lim\limits_{l\to\infty}\int_{0}^{t_2}\int_{\Omega}\ten{G}(\tilde{\theta} + \theta_{k,l},\tilde{\ten{T}}^d + \ten{T}^d_{k,l}):\tilde{\ten{T}}^d  \dxdt\\
&\le
\limsup\limits_{l\to\infty}\int_{0}^{t_2+\mu}\int_{\Omega}\ten{G}(\tilde{\theta} + \theta_{k,l},\tilde{\ten{T}}^d + \ten{T}^d_{k,l}):(\tilde{\ten{T}}^d + \ten{T}^d_{k,l})\psi_{\mu,t_2} \dxdt\\
&-\lim\limits_{l\to\infty}\int_{0}^{t_2}\int_{\Omega}\ten{G}(\tilde{\theta} + \theta_{k,l},\tilde{\ten{T}}^d + \ten{T}^d_{k,l}):\tilde{\ten{T}}^d  \dxdt\\
&\le
\limsup_{l\rightarrow\infty}\int_{0}^{t_2+\mu}\int_{\Omega}\ten{G}(\tilde{\theta} + \theta_{k,l},\tilde{\ten{T}}^d + \ten{T}^d_{k,l}):\ten{T}^d_{k,l} \ \psi_{\mu,t_2}\dxdt\\
&+\lim\limits_{l\to\infty}\int_{0}^{t_2+\mu}\int_{\Omega}\ten{G}(\tilde{\theta} + \theta_{k,l},\tilde{\ten{T}}^d + \ten{T}^d_{k,l}):\tilde{\ten{T}}^d \ \psi_{\mu,t_2}\dxdt\\
&-\lim\limits_{l\to\infty}\int_{0}^{t_2}\int_{\Omega}\ten{G}(\tilde{\theta} + \theta_{k,l},\tilde{\ten{T}}^d + \ten{T}^d_{k,l}):\tilde{\ten{T}}^d  \dxdt\\
& \leq
\int_{0}^{t_2+\mu}\int_{\Omega}\ten{\chi}_k:\ten{T}^d_k \ \psi_{\mu,t_2}\dxdt\\
&+\lim\limits_{l\to\infty}
\int_{t_2}^{t_2+\mu}\int_{\Omega}\ten{G}(\tilde{\theta} + \theta_{k,l},\tilde{\ten{T}}^d + \ten{T}^d_{k,l}):\tilde{\ten{T}}^d  \psi_{\mu,t_2}\dxdt\\
&=\int_{0}^{t_2+\mu}\int_{\Omega}\ten{\chi}_k:\ten{T}^d_k \ \psi_{\mu,t_2}\dxdt +\lim\limits_{l\to\infty}\int_{t_2}^{t_2+\mu}\int_{\Omega}
\ten{G}(\tilde{\theta} + \theta_{k,l},\tilde{\ten{T}}^d + \ten{T}^d_{k,l}):\tilde{\ten{T}}^d_k\dxdt
\end{split}\end{equation}
Passing with $\mu\to0$ yields \eqref{teza-8}. The proof is complete.
\end{proof}

To identify the weak limit~$\ten{\chi}_k$ we use the Minty-Browder trick. From the monotonicity of the function $\ten{G}(\cdot,\cdot)$ we obtain
\begin{equation}
\begin{split}
\int_{\Omega}\left(\ten{G}(\tilde{\theta} + \theta_{k,l},\tilde{\ten{T}}^d + \ten{T}^d_{k,l})  - \ten{G}(\tilde{\theta} + \theta_{k,l},\tilde{\ten{T}}^d + \ten{W}^d)\right) : &(\ten{T}_{k,l}^d - \ten{W}^d) \dx\geq 0 
\\
&
\forall \ \ten{W}^d\in L^p(0,T,L^p(\Omega,\mathcal{S}^3)).
\label{eq:zal_G1}
\end{split}
\end{equation}
Hence
\begin{equation}
\begin{split}
\int_0^T\int_{\Omega} \ten{G}(\tilde{\theta} + \theta_{k,l},\tilde{\ten{T}}^d + \ten{T}^d_{k,l}):\ten{T}_{k,l}^d \dxdt
- &\int_0^T\int_{\Omega} \ten{G}(\tilde{\theta} + \theta_{k,l},\tilde{\ten{T}}^d + \ten{T}^d_{k,l}):\ten{W}^d \dxdt
\\
- \int_0^T\int_{\Omega} \ten{G}(\tilde{\theta} + \theta_{k,l},\tilde{\ten{T}}^d + \ten{W}^d) :\ten{T}_{k,l}^d \dxdt
+ &\int_0^T\int_{\Omega} \ten{G}(\tilde{\theta} + \theta_{k,l},\tilde{\ten{T}}^d + \ten{W}^d) : \ten{W}^d \dxdt\geq 0 .
\end{split}
\label{eq:44}
\end{equation}
The pointwise convergence of $\{\theta_{k,l}\}$ implies the pointwise convergence of $\{\ten{G}(\tilde{\theta} + \theta_{k,l},\tilde{\ten{T}}^d + \ten{W}^d)\}$. The function $|\tilde{\ten{T}}^d + \ten{W}^d|^{p-1}$ belongs to $L^{p'}(0,T,L^{p'}(\Omega))$, hence the sequence $\{\ten{G}(\tilde{\theta} + \theta_{k,l},\tilde{\ten{T}}^d +\ten{W}^d)\}$ is uniformly bounded in $L^{p'}(0,T,L^{p'}(\Omega,\mathcal{S}^3))$. Then, using the Lebesgue dominated convergence theorem
we obtain that $\ten{G}(\tilde{\theta} + \theta_{k,l},\tilde{\ten{T}}^d + \ten{W}^d)\rightarrow \ten{G}(\tilde{\theta} + \theta_k,\tilde{\ten{T}}^d +\ten{W}^d)$ in $L^{p'}(0,T, L^{p'}(\Omega,\mathcal{S}^3))$ for every $\ten{W}^d\in L^p(0,T,L^p(\Omega,\mathcal{S}^3))$. 
Letting  $l\to\infty$ in \eqref{eq:44}, we get
\begin{equation}
\int_0^T \int_{\Omega} \left(\ten{\chi}_k -\ten{G}(\tilde{\theta} + \theta_{k},\tilde{\ten{T}}^d + \ten{W}^d)\right):( \ten{T}_k^d - \ten{W}^d) \dxdt \geq 0
\qquad \forall \ \ten{W}^d\in L^p(0,T,L^p(\Omega,\mathcal{S}^3)),
\end{equation}
and taking $\ten{W}^d = \ten{T}^d_k - \lambda \ten{U}^d$, where $\ten{U}^d\in L^p(0,T,L^p(\Omega,\mathcal{S}^3))$ and $\lambda>0$, then
\begin{equation}
\begin{split}
\int_0^T \int_{\Omega}\left( \ten{\chi}_k -\ten{G}(\tilde{\theta} + \theta_{k},\tilde{\ten{T}}^d + \ten{T}_k^d - \lambda \ten{U}^d)\right):(  \lambda \ten{U}^d) \dxdt \geq 0 \quad \forall \ \ten{U}^d\in L^p(0,T,L^p(\Omega,\mathcal{S}^3))
\end{split}
\end{equation}
hence
\begin{equation}
\begin{split}
\int_0^T \int_{\Omega} \left(\ten{\chi}_k -\ten{G}(\tilde{\theta} + \theta_{k},\tilde{\ten{T}}^d + \ten{T}_k^d - \lambda \ten{U}^d)\right): \ten{U}^d \dxdt \geq 0 \quad \forall \ \ten{U}^d\in L^p(0,T,L^p(\Omega,\mathcal{S}^3)).
\end{split}
\end{equation}
Letting  $\lambda\to0$ we obtain 
\begin{equation}
\int_0^T \int_{\Omega} \left(\ten{\chi}_k -\ten{G}(\tilde{\theta} + \theta_{k},\tilde{\ten{T}}^d + \ten{T}_k^d )\right): \ten{U}^d \dxdt \geq 0 \qquad \forall\  \ten{U}^d\in L^p(0,T,L^p(\Omega,\mathcal{S}^3)).
\end{equation}
Choosing now $\lambda<0$  we obtain the opposite inequality and hence
\begin{equation}
\int_0^T \int_{\Omega} \left(\ten{\chi}_k -\ten{G}(\tilde{\theta} + \theta_{k},\tilde{\ten{T}}^d + \ten{T}_k^d )\right): \ten{U}^d \dxdt = 0 \qquad \forall\ \ten{U}^d\in L^p(0,T,L^p(\Omega,\mathcal{S}^3)).
\end{equation}
Thus
\begin{equation}
\ten{\chi}_k =\ten{G}(\tilde{\theta} + \theta_{k},\tilde{\ten{T}}^d + \ten{T}_k^d )\quad \mbox{a.e. in}\ (0,T)\times\Omega.
\label{82}\end{equation}
Consequently for every $k\in\mathbb{N}$ 
$$\ten{G}(\tilde{\theta} + \theta_{k,l},\tilde{\ten{T}}^d + \ten{T}_{k,l}^d)\rightharpoonup \ten{G}(\tilde{\theta} + \theta_k,\tilde{\ten{T}}^d + \ten{T}_k^d)\quad \mbox{in} \ L^{p'}(0,T,L^{p'}(\Omega,\mathcal{S}^3))$$ 
as $l\rightarrow \infty$. 

\begin{lemat}
For each $k\in\mathbb{N}$ it holds
\begin{equation}\begin{split}
\lim\limits_{l\to\infty}\int_0^T\int_\Omega&\ten{G}(\tilde{\theta} + \theta_{k,l},\tilde{\ten{T}}^d + \ten{T}_{k,l}^d):(\tilde{\ten{T}}^d + \ten{T}_{k,l}^d )\dxdt\\&
= \int_0^T\int_\Omega\ten{G}(\tilde{\theta} + \theta_k,\tilde{\ten{T}}^d + \ten{T}_k^d):(\tilde{\ten{T}}^d + \ten{T}_k^d)\dxdt.
\end{split}\end{equation}
\end{lemat}

\begin{proof}
Using monotonicity of the function $\ten{G}(\cdot,\cdot)$
\begin{equation}
\begin{split}
0 &\leq 
\int_0^T\int_{\Omega}
\left(
\ten{G}(\tilde{\theta} + \theta_{k,l},\tilde{\ten{T}}^d + \ten{T}_{k,l}^d) - \ten{G}(\tilde{\theta} + \theta_{k,l},\tilde{\ten{T}}^d + \ten{T}_{k}^d)\right):(\ten{T}_{k,l}^d - \ten{T}_k^d) \dxdt
\\
&=\int_0^T\int_{\Omega}
\ten{G}(\tilde{\theta} + \theta_{k,l},\tilde{\ten{T}}^d + \ten{T}_{k,l}^d):(\ten{T}_{k,l}^d - \ten{T}_k^d) - \ten{G}(\tilde{\theta} + \theta_{k,l},\tilde{\ten{T}}^d + \ten{T}_{k}^d):(\ten{T}_{k,l}^d - \ten{T}_k^d) \dxdt .
\end{split}
\label{eq:50}
\end{equation}
Passing with $l$ to $\infty$ we get that the second term from \eqref{eq:50} converges to zero. Furthermore, using Lemma \ref{lm:8}
\begin{equation}
\begin{split}
0 &\leq 
\limsup_{l\rightarrow \infty}
\int_0^T\int_{\Omega}
\ten{G}(\tilde{\theta} + \theta_{k,l},\tilde{\ten{T}}^d + \ten{T}_{k,l}^d):(\ten{T}_{k,l}^d + \tilde{\ten{T}}^d - \tilde{\ten{T}}^d -\ten{T}_{k}^d) \dxdt
\\
&= 
\limsup_{l\rightarrow \infty}
\int_0^T\int_{\Omega}
\ten{G}(\tilde{\theta} + \theta_{k,l},\tilde{\ten{T}}^d + \ten{T}_{k,l}^d):(\tilde{\ten{T}}^d + \ten{T}_{k,l}^d) \dxdt
\\
& \quad
- \lim_{l\rightarrow \infty}
\int_0^T\int_{\Omega}
\ten{G}(\tilde{\theta} + \theta_{k,l},\tilde{\ten{T}}^d + \ten{T}_{k,l}^d):(\tilde{\ten{T}}^d + \ten{T}_{k}^d) \dxdt 
\le 0 .
\end{split}
\end{equation}
Hence 
\begin{equation}
0 = \lim_{l\rightarrow\infty} 
\int_0^T\int_{\Omega}
\left(
\ten{G}(\tilde{\theta} + \theta_{k,l},\tilde{\ten{T}}^d + \ten{T}_{k,l}^d) - \ten{G}(\tilde{\theta} + \theta_{k,l},\tilde{\ten{T}}^d + \ten{T}_{k}^d)\right):(\ten{T}_{k,l}^d - \ten{T}_k^d) \dxdt ,
\end{equation}
and
\begin{equation}
\begin{split}
\lim_{l\rightarrow \infty}
\int_0^T\int_{\Omega}
\ten{G}(\tilde{\theta} + \theta_{k,l},\tilde{\ten{T}}^d + \ten{T}_{k,l}^d):&(\tilde{\ten{T}}^d + \ten{T}_{k,l}^d) \dxdt
\\&= 
\int_0^T\int_{\Omega}
\ten{G}(\tilde{\theta} + \theta_{k},\tilde{\ten{T}}^d + \ten{T}_{k}^d):(\tilde{\ten{T}}^d + \ten{T}_{k}^d) \dxdt ,
\end{split}
\nonumber
\end{equation}
which completes the proof.
\end{proof}

Hence now we can also pass to the limit in the heat equation, namely we obtain for all 
$\phi\in C^\infty([0,T]\times\Omega)$
\begin{equation}
\begin{split}
-\int_0^T\int_{\Omega} \theta_k\phi_t \dxdt -
\int_{\Omega} \theta_k(x,0)\phi(x,0) \dx    + 
\int_0^T\int_{\Omega}  \nabla\theta_k\cdot\nabla\phi  \dxdt  \\= 
\int_0^T\int_{\Omega} \mathcal{T}_k\left((\ten{T}^d_{k}+\tilde{\ten{T}}^d):\ten{G}(\theta_{k}+\tilde{\theta},\ten{T}^d_{k}+\tilde{\ten{T}}^d)\right)\phi \dxdt,
\\
\label{eq:after_limit_l_2}
\end{split}
\end{equation}

\subsection{Limit passage $k\to\infty$ } 
\label{sec:8}



We start this section with  considerations on the sequence of temperatures. 
We are using the 
 result of Boccardo and Galllou\"{e}t \cite{Boccardo} for parabolic equation with only integrable data and Dirichlet boundary conditions. 
 Since our studies concern the problem with Neumann boundary conditions, we include the modification of their result in the  Appendix A. 
 Consequently, we conclude for each $1<q<\frac{5}{4}$
 \begin{equation}
 \theta_k \rightharpoonup \theta  \mbox{ weakly in } L^q(0,T,W^{1,q}(\Omega)). \\
 \end{equation}
Moreover, the uniform estimates  from  the previous sections allow to conclude that at least for a subsequence the  following holds
\begin{equation}
\begin{array}{cl}
\theta_k \rightarrow \theta & \mbox{ a.e. in }\Omega\times(0,T),\\
\vc{u}_k \rightharpoonup  \vc{u} & \mbox{ weakly in } L^{p'}(0,T,W^{1,p'}_0(\Omega,\mathbb{R}^3)),\\
\ten{T}_k \rightharpoonup \ten{T} & \mbox{ weakly in }  L^2(0,T,L^2(\Omega,\mathcal{S}^3)),\\
\ten{T}^d_k \rightharpoonup \ten{T}^d & \mbox{ weakly in }  L^p(0,T,L^p(\Omega,\mathcal{S}^3_d)),\\
\ten{G}(\tilde{\theta} + \theta_k,\tilde{\ten{T}}^d +\ten{T}_k^d) \rightharpoonup \ten{\chi} & \mbox{ weakly in }  L^{p'}(0,T,L^{p'}(\Omega,\mathcal{S}^3_d)), \\
(\ten{\varepsilon}^{\bf p}_k)_t \rightharpoonup (\ten{\varepsilon}^{\bf p})_t & \mbox{ weakly in }  L^{p'}(0,T,L^{p'}(\Omega,\mathcal{S}^3_d)).
\end{array}
\end{equation}
Consequently, passing to the limit in \eqref{limit1}, \eqref{65} we obtain
\begin{equation}\label{limit1a}
\begin{split}
\int_0^T\int_{\Omega}\ten{T}:\nabla\vc{\varphi} \dxdt =0
\end{split}
\end{equation}
\begin{equation}\label{65a}
\int_0^T\int_{\Omega}(\ten{\varepsilon}^{\bf p})_t : \ten{\psi}\dxdt = 
\int_0^T\int_{\Omega}\ten{\chi}  :  \ten{\psi}\dxdt
 \end{equation}
 for all $ \ten{\varphi}\in C^\infty([0,T],L^2(\Omega,\mathcal{S}^3))$ and then also for all 
$ \ten{\varphi}\in L^2(0,T;L^2(\Omega,\mathcal{S}^3))$ and for all $\ten{\psi}\in L^p(0,T;L^p(\Omega,\mathcal{S}^3))$
To characterize the limit $\ten{\chi}$ and pass to the limit in the heat equation we follow the similar lines as in the limit passage with $l\to\infty$. 

\begin{lemat}
The following inequality holds for the solution of approximate systems.
\begin{equation}
\limsup_{k\rightarrow\infty}\int_{0}^{t_2}\int_{\Omega}\ten{G}(\tilde{\theta} + \theta_{k},\tilde{\ten{T}}^d + \ten{T}^d_{k}):\ten{T}^d_k \dxdt \leq
\int_{0}^{t_2}\int_{\Omega}\ten{\chi}:\ten{T}^d \dxdt .
\label{jedna_nierownosc_1}
\end{equation}
\end{lemat}

\begin{proof}
Due to \eqref{82} we can rewrite \eqref{gran} as follows
\begin{equation}\label{do-l}
\frac{d}{dt} \mathcal{E}(\ten{\varepsilon}(\vc{u}_{k}) , \ten{\varepsilon}^{\bf p}_{k}) 
 = 
-
\int_{\Omega}\ten{G}(\tilde{\theta} + \theta_{k},\tilde{\ten{T}}^d + \ten{T}^d_{k}):\ten{T}^d_{k}\dx.
\end{equation}
We multiply the above identity by $\psi_{\mu,t_2}$ given by formula \eqref{psi-mu} and integrate over $(0,T)$. 
Passing to  the limit $k\to\infty$ we proceed in the same manner as in the proof of Lemma~\ref{lm:8} and obtain
\begin{equation}\label{mu4a}
\begin{split}
\liminf\limits_{k\to\infty}\int_{0}^{T}
\frac{d}{d\tau}& \mathcal{E}(\ten{\varepsilon}(\vc{u}_{k}) , \ten{\varepsilon}^{\bf p}_{k}) \,\psi_{\mu,t_2}\dt\\
&=\liminf\limits_{k\to\infty}\frac{1}{\mu}\int_{t_2}^{t_2+\mu}\mathcal{E}(\ten{\varepsilon}(\vc{u}_{k}) , \ten{\varepsilon}^{\bf p}_{k}) \dt-
\lim\limits_{k\to\infty}\mathcal{E}(\ten{\varepsilon}(\vc{u}_{k}(0)) , \ten{\varepsilon}^{\bf p}_{k}(0))\\
&\ge \frac{1}{\mu}\int_{t_2}^{t_2+\mu}\mathcal{E}(\ten{\varepsilon}(\vc{u}_{k}(t)) , \ten{\varepsilon}^{\bf p}_{k}(t)) \dt-
\mathcal{E}(\ten{\varepsilon}(\vc{u}(0)) , \ten{\varepsilon}^{\bf p}(0)).
\end{split}\end{equation}

For the final step of the proof of the lemma we need to show that  the energy equality holds. Contrary to the case of previous section, we cannot use the time derivative of the limit, namely $\vc{\varepsilon}(\vc{u})_t$ as the test function. Although we shall mollifty with respect to time, but  the regularity with respect to space is not sufficient since possibly $p'< 2$. Therefore we proceed differently. We use an approximate sequence as a test function in the limit identity. Indeed, we take in \eqref{limit1a} 
the test function $\vc{\varphi}=(\ten{\varepsilon}(\vc{u}_k)*\eta_\epsilon)_t\mathbf{1}_{(t_1,t_2)})*\eta_{\epsilon}$, 
where again $\eta_\epsilon$ is a standard mollifier and we mollify with respect to time
\begin{equation}
\int_{t_1}^{t_2} \int_{\Omega}\ten{D}(\ten{\varepsilon}(\vc{u}) - \ten{\varepsilon}^{\bf p})*\eta_{\epsilon}: (\ten{\varepsilon}(\vc{u}_k)*\eta_{\epsilon})_t \dxdt =
0.
\label{pierwsze_r2}
\end{equation}
Then we use the approximate equation  \eqref{65} with a test function $\ten{\psi}=(\ten{T}_k^d*\eta_{\epsilon}\mathbf{1}_{(t_1,t_2)})*\eta_{\epsilon}$. In a consequence we obtain \eqref{drugie_r2}, which together with \eqref{82} yields
\begin{equation}
\int_{t_1}^{t_2}\int_{\Omega}(\ten{\varepsilon}^{\bf p}_{k}*\eta_{\epsilon})_t:\ten{T}*\eta_{\epsilon} \dxdt =
\int_{t_1}^{t_2}\int_{\Omega}\ten{G}(\tilde{\theta}+\theta_k,\tilde{\ten{T}}^d+\ten{T}^d_k)*\eta_{\epsilon}:\ten{T}*\eta_{\epsilon} \dxdt .
\label{drugie_r2a}
\end{equation}
Products in \eqref{drugie_r2a} are well defined, since for the matrices $\ten{A}\in\mathcal{S}^3_d$ and $\ten{B}\in\mathcal{S}^3$ the equivalence $\ten{A}:\ten{B}^d=\ten{A}:\ten{B}$ holds and tensor $\ten{T}^d$ belongs to $L^{p'}(0,T,L^{p'}(\Omega,\mathcal{S}^3))$.
Subtracting \eqref{drugie_r2a} from \eqref{pierwsze_r2} we get
\begin{equation}
\int_{t_1}^{t_2}\int_{\Omega}\ten{T}*\eta_{\epsilon}:(\ten{\varepsilon}(\vc{u}_k) - \ten{\varepsilon}^{\bf p}_k)_t*\eta_{\epsilon} \dxdt=
- 
\int_{t_1}^{t_2}\int_{\Omega}\ten{G}(\tilde{\theta}+\theta_k,\tilde{\ten{T}}^d + \ten{T}^d_k)*\eta_{\epsilon}:\ten{T}^d*\eta_{\epsilon} \dxdt .
\label{granica_k_ptrzed}
\end{equation}
For every $\epsilon>0$ the sequence $\{(\ten{\varepsilon}(\vc{u}_k) - \ten{\varepsilon}^{\bf p}_k)_t*\eta_{\epsilon}\}$ belongs to $L^2(0,T,L^2(\Omega,\mathcal{S}^3))$ and is uniformly bounded in $L^2(0,T,L^2(\Omega,\mathcal{S}^3))$, hence we pass to the limit with $k\rightarrow\infty$ and we obtain


\begin{equation}
\int_{t_1}^{t_2}\int_{\Omega}\ten{T}*\eta_{\epsilon}:(\ten{\varepsilon}(\vc{u}) - \ten{\varepsilon}^{\bf p})_t*\eta_{\epsilon} \dxdt=
- 
\int_{t_1}^{t_2}\int_{\Omega}\ten{\chi}*\eta_{\epsilon}:\ten{T}^d*\eta_{\epsilon} \dxdt .
\nonumber
\end{equation}
Using the properties of convolution we get
\begin{equation}
\int_{\Omega}\ten{T}*\eta_{\epsilon}:(\ten{\varepsilon}(\vc{u}) - \ten{\varepsilon}^{\bf p})*\eta_{\epsilon} \dx \Big|_{t_1}^{t_2}=
- 
\int_{t_1}^{t_2}\int_{\Omega}\ten{\chi}*\eta_{\epsilon}:\ten{T}^d*\eta_{\epsilon}*\eta_{\delta} \dxdt ,
\nonumber
\end{equation}
and finally passing to the limit with $\epsilon\rightarrow 0$ and then with $t_1\to0$
\begin{equation}
\int_{\Omega}\ten{D}(\ten{\varepsilon}(\vc{u}) - \ten{\varepsilon}^{\bf p}):(\ten{\varepsilon}(\vc{u}) - \ten{\varepsilon}^{\bf p}) \dx \Big|_{0}^{t_2}=
- 
\int_{0}^{t_2}\int_{\Omega}\ten{\chi}:\ten{T}^d \dxdt .
\label{granica_k}
\end{equation}
We multiply \eqref{granica_k} by $\frac{1}{\mu}$ and integrate over $(t_2,t_2+\mu)$ and proceed now in the same manner as in the proof of Lemma~\ref{lm:8} to complete the proof.

\end{proof}

Using the Minty-Browder trick to identify the weak limit $\ten{\chi}$ and the same argumentation as in the previous section, we obtain that 
\begin{equation}
\ten{G}(\tilde{\theta} + \theta_k,\tilde{\ten{T}}^d + \ten{T}_k):(\tilde{\ten{T}}^d + \ten{T}^d_{k})\rightharpoonup \ten{G}(\tilde{\theta} + \theta,\tilde{\ten{T}} +\ten{T}):(\tilde{\ten{T}}^d + \ten{T}^d) 
\quad\mbox{ in }L^1(0,T,L^1(\Omega)).
\end{equation}
Furthermore 
\begin{equation}
\mathcal{T}_k\Big(\ten{G}(\tilde{\theta} + \theta_k,\tilde{\ten{T}}^d + \ten{T}_k):(\tilde{\ten{T}}^d + \ten{T}^d_{k})\Big)\rightharpoonup \ten{G}(\tilde{\theta} + \theta,\tilde{\ten{T}} +\ten{T}):(\tilde{\ten{T}}^d + \ten{T}^d)
\end{equation}
in $L^1(0,T,L^1(\Omega))$. Using convergences presented above we pass to the limit with $k\rightarrow \infty$ in the equations \eqref{limit1} and \eqref{eq:after_limit_l_2}, include the previously removed boundary and volume force term and obtain
\begin{equation}
\begin{split}
\int_0^T\int_{\Omega}\Big(\tilde{\ten{T}} + \ten{T} \Big):\nabla\vc{\varphi} \dxdt 
&= \int_0^T\int_{\Omega}\vc{f}\cdot \vc{\varphi} \dxdt ,
\end{split}
\end{equation}
where
\begin{equation}
\ten{T}=\ten{D}(\ten{\varepsilon}(\vc{u}) - \ten{\varepsilon}^{\bf p} )
\qquad
\tilde{\ten{T}} = \ten{\varepsilon}(\tilde{\vc{u}}),
\end{equation}
and
\begin{equation}
\begin{split}
-\int_0^T\int_{\Omega} (\tilde{\theta} + \theta)\phi_t \dxdt &-
\int_{\Omega} (\tilde{\theta}_0(x) + \theta_0(x))\phi(x,0) \dx  
\\ 
+ \int_0^T\int_{\Omega}  \nabla(\tilde{\theta} +\theta)\cdot\nabla\phi  \dxdt 
&- \int_0^T\int_{\partial\Omega}g_{\theta}\phi  \ds \dt 
\\
&= 
\int_0^T\int_{\Omega} (\tilde{\ten{T}}^d + \ten{T}^d):\ten{G}(\tilde{\theta}+\theta,\tilde{\ten{T}}^d+\ten{T}^d)\phi \dxdt,
\end{split}
\end{equation}
and
\begin{equation}
\ten{\varepsilon}^{\bf p}(x,t) = \ten{\varepsilon}^{\bf p}_0(x) + \int_0^t \ten{G}(\tilde{\theta} + \theta,\tilde{\ten{T}}^d + \ten{T}^d)\dtau ,
\end{equation}
what completes the proof of Theorem \ref{thm:main2}.

\begin{appendix}
\section{}


Let  $\mathcal{T}_k(\cdot)$ be a standard truncation operator defined in \eqref{Tk}. In \cite{Boccardo}, the authors showed the existence of  solutions for the heat equation with Dirichlet boundary conditions. The current section is devoted to the existence proof to the problem with Neumann boundary conditions. Two dimensional case was considered in \cite{1240/THESES}.  


We consider the sequence of the heat equations with boundary and initial conditions and with the right hand side of equation in the form 
\begin{equation}
f_k=\mathcal{T}_k\left((\tilde{\ten{T}}^d + \ten{T}_k^d) : \ten{G}(\tilde{\theta} + \theta_k,\tilde{\ten{T}}^d + \ten{T}_k^d)\right)
\end{equation}
 which for every $k\in\mathbb{N}$ belongs to $L^2(0,T,L^2(\Omega))$ and moreover is uniformly bounded $\|f_k\|_{L^1(0,T,L^1(\Omega))}\leq C$ and $f_k\rightarrow f$ in $L^1(0,T,L^1(\Omega))$ as $k\to\infty$. Additionally, we have $\mathcal{T}_k(\theta_0)\in L^2(\Omega)$, $\|\mathcal{T}_k(\theta_0)\|_{L^1(\Omega)}\leq \|\theta_0\|_{L^1(\Omega)}$ and $\mathcal{T}_k(\theta_0)\rightarrow \theta_0$ in $L^1(\Omega)$. To simplify the notation in the remaining part  of the Appendix we denote  $\Omega\times(0,T)$ by $Q$.
 Let us consider the following problem
\begin{equation}
\left\{
\begin{array}{cc}
(\theta_k)_t - \Delta\theta_k = f_k
& \mbox{ in } \Omega\times (0,T),
\\
\frac{\partial\theta_k}{\partial\vc{n}}=0
& \mbox{ in }  \partial\Omega\times(0,T),
\\
\theta_k(\cdot,0)=\mathcal{T}_k(\theta_0)
&  \mbox{ on } \quad \Omega .
\end{array}
\right.
\label{ukla_para_n}
\end{equation}
and its weak formulation
\begin{equation}
\begin{split}
\int_{0}^T \int_{\Omega}\theta_k\varphi_t \dxdt +
\int_{0}^T\int_{\Omega}\nabla\theta_k\cdot\nabla\varphi \dxdt 
=
\int_{0}^T\int_{\Omega}f_k\varphi \dxdt +\int_\Omega \theta_0\varphi(0)\dx,
\end{split}\label{slab}
\end{equation}
holding for all $\varphi\in L^q(0,T;W^{1,q}(\Omega))$.

\begin{lemat}
The sequence of approximate  solutions to the  heat equation \eqref{ukla_para_n} is uniformly bounded in the space $L^q(0,T,W^{1,q}(\Omega))$ for $q<\frac{2 (N + 1) - N}{N + 1}$ ($q<\frac{5}{4}$ in tree dimensional case $N = 3$).  
\label{unif_boun_LqLq}
\end{lemat}


\begin{proof}
We define the special truncation function $\psi_m(\cdot)$ for every $m\in\mathbb{N}$:
\begin{equation}
\psi_m(s)=\left\{
\begin{array}{ccl}
1 & \mbox{if} & s\geq m+1, \\
s-m & \mbox{if} & m+1\geq s\geq m, \\
0 & \mbox{if} & |s|\leq m, \\
s+m & \mbox{if} & s\geq m+1, \\
-1 & \mbox{if} & s\leq -m-1. \\
\end{array}
\right.
\end{equation}
Using in \eqref{slab}  the test function $\psi_m(\theta_k)$  we obtain

\begin{equation}
\begin{split}
\int_{0}^T \int_{\Omega}(\Psi_m(\theta_k))_t \dxdt +
\int_{0}^T\int_{\Omega}\nabla\theta_k\cdot\nabla\psi_m(\theta_k) \dxdt 
=\int_{0}^T\int_{\Omega}f_k\psi_m(\theta_k) \dxdt ,
\end{split}
\end{equation}
where $\Psi_m(s)=\int_0^s\psi_m(\sigma)d\sigma$. Thus
\begin{equation}
\begin{split}
\int_{\Omega}\Psi_m(\theta_k)(T) \dx +
\int_{0}^T\int_{\Omega}\nabla\theta_k\cdot\nabla\psi_m(\theta_k) \dxdt =
\int_{0}^T\int_{\Omega}f_k\psi_m(\theta_k) \dxdt +
\int_{\Omega}\Psi_m(\mathcal{T}_k(\theta_0)) \dx .
\nonumber
\end{split}
\end{equation}
The terms on the right side of the above equation  can be estimated as follows 
\begin{equation}
\begin{split}
\int_{0}^T\int_{\Omega}f_k\psi_m(\theta_k) \dxdt & \leq 
\|f\|_{L^1(0,T,L^1(\Omega))},
\nonumber \\
\int_{\Omega}\Psi_m(\mathcal{T}_k(\theta_0)) \dx & \leq
\|\theta_0\|_{L^1(\Omega)},
\nonumber
\end{split}
\end{equation}
for every $k,m\in\mathbb{N}$. Additionally, $\int_{\Omega}\Psi_m(\theta_k)(T) dx$ is nonnegative. Hence,
\begin{equation}
\begin{split}
\int_{B_m}|\nabla\theta_k|^2 \dxdt =
\int_{0}^T\int_{\Omega}\nabla\theta_k\cdot\nabla\psi_m(\theta_k) \dxdt \leq
\|f\|_{L^1(0,T,L^1(\Omega))} +
\|\theta_0\|_{L^1(\Omega)},
\nonumber
\end{split}
\end{equation}
where the set $B_m=\left\{ (x,t)\in \Omega\times(0,T): m\leq \theta_k(x,t)\leq m+1 \right\}$. Now let $q \leq \frac{2(N+1) - N}{N + 1}$ and  $r = \frac{N+1}{N} q$ (in our case $q <\frac{5}{4}$ and $r=\frac{4}{3}q$). Using the H\"older inequality we obtain 
\begin{equation}
\begin{split}
\int_{B_m}|\nabla \theta_k|^q \dxdt
&\leq
\left(\int_{B_m}|\nabla\theta_k|^{q\frac{2}{q}} \dxdt
\right)^{\frac{q}{2}}
\left(\int_{B_m}1^{\frac{2}{2-q}} \dxdt \right)^{1-\frac{q}{2}}
\\
& \leq \left(\int_{B_m}|\nabla\theta_k|^2 \dxdt \right)^{\frac{q}{2}}
\left(\int_{B_m} \dxdt \right)^{1-\frac{q}{2}} 
\\
& \leq
c_3 \left(\int_{B_m}\frac{|\theta_k|^r}{m^r} \dxdt \right)^{1-\frac{q}{2}} 
\\
& \leq
c_3 \left(\int_{B_m}|\theta_k|^r \dxdt \right)^{1-\frac{q}{2}}\frac{1}{m^{\frac{r(2-q)}{2}}} 
\\
& \leq
c_3 \left(\int_{B_m}|\theta_k|^r \dxdt \right)^{1-\frac{q}{2}}\left(\frac{1}{m^{\frac{r(2-q)}{q}}}\right)^{\frac{q}{2}} .
\nonumber 
\end{split}
\end{equation}
Then 
\begin{equation}
\begin{split}
\int_Q|\nabla\theta_k|^q \dxdt &\leq
c_4(n_0)+c_3\sum_{m=n_0}^{\infty}\left(\int_{B_m}|\theta_k|^r \dxdt \right)^{1-\frac{q}{2}}\left(\frac{1}{m^{\frac{r(2-q)}{q}}}\right)^{\frac{q}{2}} 
\\
& \leq
c_4(n_0)+c_3\left(\sum_{m=n_0}^{\infty}\int_{B_m}|\theta_k|^r \dxdt \right)^{1-\frac{q}{2}}\left(\sum_{m=n_0}^{\infty}\frac{1}{m^{\frac{r(2-q)}{q}}}\right)^{\frac{q}{2}}
\\
& \leq
c_4(n_0)+c_3\left(\int_{Q}|\theta_k|^r \dxdt \right)^{1-\frac{q}{2}}\left(\sum_{m=n_0}^{\infty}\frac{1}{m^{\frac{r(2-q)}{q}}}\right)^{\frac{q}{2}},
\end{split}
\label{eq:dobre_ograniczenie}
\end{equation}
where $c_4(n_0)=\int_{\{(x,t):|\theta_k(x,t)|\leq n_0\}}|\nabla \theta_k|^q \dxdt $. Using the H\"older inequality we observe that $c_4(n_0)$ is bounded by the terms $\|f\|_{L^1(0,T,L^1(\Omega))}$, $\|u_0\|_{L^1(\Omega)}$ and the measure of the set $Q$. Furthermore, $\frac{r(2-q)}{q}>1$ and $\sum_{m=n_0}^{\infty}m^{-\frac{r(2-q)}{q}}$ is summable.
Using the interpolation inequality for $\|\theta_k\|_{L^q(\Omega)}$ we obtain
\begin{equation}
\begin{split}
\|\theta_k\|_{L^q(\Omega)}\leq
\|\theta_k\|_{L^1(\Omega)}^s
\|\theta_k\|_{L^{q^*}(\Omega)}^{1-s},
\end{split}
\label{eq:int_1}
\end{equation}
where $q^*=\frac{Nq}{N - q}$ ($=\frac{3q}{3 - q}$) and $\frac{1}{q}=\frac{s}{1}+\frac{1-s}{q^*}$. After simple calculations we get that $1-s=\frac{1 - q}{1 - q^*}\frac{q^*}{q}$ (and $0<s<1$). In Lemma \ref{LinftyL1} we showed that $\|\theta_k\|_{L^1(\Omega)}$ is uniformly bounded, hence
\begin{equation}
\begin{split}
\int_0^T\int_{\Omega}|\theta_k|^q \dxdt \leq
C \int_0^T\|\theta_k\|_{L^{q^*}(\Omega)}^{(1-s)q} \dt \leq
C \int_0^T\|\theta_k\|_{L^{q^*}(\Omega)}^{\frac{1 - q}{1 - q^*}q^*} \dt .
\nonumber 
\end{split}
\end{equation}
Using the H\"older inequality we obtain
\begin{equation}
\begin{split}
\int_0^T\int_{\Omega}|\theta_k|^q \dxdt &\leq
C \int_0^T\|\theta_k\|_{L^{q^*}(\Omega)}^{\frac{1 - q}{1 - q^*}q^*} \dt 
\\
& \leq
C \left(\int_0^T\|\theta_k\|_{L^{q^*}(\Omega)}^{\frac{1 - q}{1 - q^*}q^*\frac{q^* - 1}{q - 1}\frac{q}{q^*}} \dt \right)^{\frac{q - 1}{q^*-1}\frac{q^*}{q}}
\\
& =
C \left(\int_0^T\|\theta_k\|_{L^{q^*}(\Omega)}^q \dt \right)^{\frac{q - 1}{q^*-1}\frac{q^*}{q}}.
\nonumber 
\end{split}
\end{equation}
Let us notice that the exponent $\frac{q - 1}{q^*-1}\frac{q^*}{q}=\frac{N(q-1)}{N(q-1)+q}<1$.
Using the interpolation inequality for $\|\theta_k\|_{L^r(\Omega)}$ we get
\begin{equation}
\begin{split}
\|\theta_k\|_{L^r(\Omega)}\leq 
\|\theta_k\|^s_{L^1(\Omega)}
\|\theta_k\|^{1-s}_{L^{q^*}(\Omega)},
\label{eq:int_2}
\end{split}
\end{equation}
where $\frac{1}{r}=\frac{s}{1}+\frac{1-s}{q^*}$. The parameters $s$ are different in each of the interpolation inequalities \eqref{eq:int_1} and \eqref{eq:int_2}. Simple calculations yield that $1-s=\frac{1-r}{1-q^*}\frac{q^*}{r}$. By Lemma \ref{LinftyL1} we conclude that
\begin{equation}
\begin{split}
\|\theta_k\|^r_{L^r(0,T,L^r(\Omega))} & \leq
\int_0^T\|\theta_k\|^r_{L^r(\Omega)} \dt 
\\ 
& \leq
\int_0^T \|\theta_k\|^{sr}_{L^1(\Omega)}\|\theta_k\|^{\frac{1-r}{1-q^*}\frac{q^*}{r}r}_{L^{q^*}(\Omega)} \dt 
\\
& \leq
C \int_0^T \|\theta_k\|^q_{L^{q^*}(\Omega)} \dt =
C \|\theta_k\|^q_{L^q(0,T,L^{q^*}(\Omega))} .
\end{split}
\label{eq:drugie_dobre}
\end{equation}
The Sobolev embedding theorem implies that 
\begin{equation}
\begin{split}
\|\theta_k\|^q_{L^q(0,T,L^{q^*}(\Omega))}=
\int_0^T \left( \int_{\Omega}|\theta_k|^{q^*} \dx \right)^{\frac{q}{q^*}} \dt \leq
C\left(\int_0^T\int_{\Omega}|\theta_k|^q \dxdt + \int_0^T\int_{\Omega}|\nabla\theta_k|^q \dxdt \right).
\nonumber
\end{split}
\end{equation}
Using the previous inequalities we obtain
\begin{equation}
\begin{split}
\|\theta_k\|^q_{L^q(0,T,L^{q^*}(\Omega))}
& \leq
C \|\theta_k\|_{L^q(0,T,L^{q^*}(\Omega))}^{\frac{q-1}{q^*-1}\frac{q^*}{q}}+
c_4(n_0)+D\left(\int_{Q}|\theta_k|^r \dxdt \right)^{1-\frac{q}{2}}
\\
& \leq C \|\theta_k\|_{L^q(0,T,L^{q^*}(\Omega))}^{\frac{q-1}{q^*-1}\frac{q^*}{q}}+
c_4(n_0)+D\|\theta_k\|_{L^q(0,T,L^{q^*}(\Omega))}^{q\frac{2-q}{2}} 
\nonumber
\end{split}
\end{equation}
and $\frac{q-1}{q^*-1}\frac{q^*}{q}<1$ and $q\frac{2-q}{2}<q$, so we have the uniform boundedness 
\begin{equation}
\begin{split}
\|\theta_k\|^q_{L^q(0,T,L^{q^*}(\Omega))}
\leq
C,
\nonumber
\end{split}
\end{equation}
and from the previous inequalities we get the uniform boundedness of the sequence $\{\theta_k\}$ in the space 
$L^q(0,T,L^{q^*}(\Omega))$. Using this uniform boundedness and inequalities \eqref{eq:dobre_ograniczenie} and \eqref{eq:drugie_dobre} we get the uniform boundedness of the sequence $\{\theta_k\}$ in the spaces $L^q(0,T,W^{1,q}(\Omega))$, which completes the proof.
\end{proof}


\begin{lemat}
The sequence $\{\nabla\theta_k\}$ converges strongly to $\nabla \theta$ in $L^1(0,T,L^1(\Omega))$.
\label{zbieznosc_nabla_teta}
\end{lemat}

\begin{proof}
Let $\varphi$ be such that, for $\varepsilon > 0$ fixed. Let us define a test function
\begin{equation}
\varphi(s) = 
\left\{
\begin{array}{ll}
\varepsilon & s > \varepsilon, 
\\
 s & |s| \leq \varepsilon, 
 \\
-\varepsilon & s < -\varepsilon.
\end{array}
\right.
\end{equation} 
Subtracting equation \eqref{ukla_para_n} with function on right side $f_n$ and $f_m$, and using the test function $\varphi(\theta_n -\theta_m)$ we obtain
\begin{equation}
\begin{split}
\int_{\Omega}\Phi(\theta_n - &\theta_m)(T) \dx +
\int_{D_{n,m,\varepsilon}}|\nabla(\theta_n-\theta_m)|^2 \dxdt =
\nonumber \\
& \int_0^T\int_{\Omega}(f_n-f_m)\varphi(\theta_n-\theta_m) \dxdt +
\int_{\Omega}\Phi(\mathcal{T}_n(\theta_0)- \mathcal{T}_m(\theta_0)) \dx ,
\nonumber 
\end{split}
\end{equation}
where $\Phi(s)=\int_0^s\varphi(\tau)d\tau$ and $D_{n,m,\varepsilon}=\{ (x,t) \in \Omega\times (0,T): |\theta_n(x,t)-\theta_m(x,t)|\leq \varepsilon \}$. The sequence $\mathcal{T}_k(\theta_0)$ is convergent to $\theta_0$ in $L^1(\Omega)$, hence, we can find $n_0$ such that for every $n$, $m$ greater than $n_0$ we have $\int_{\Omega}\Phi(T_n(\theta_0)-T_m(\theta_0))<\varepsilon$. The function $\Phi$ is nonnegative and the right hand side of the equation above is bounded ($\|f_n\|_{L^1(0,T,L^1(\Omega))}\leq \|f\|_{L^1(0,T,L^1(\Omega))}=:B$), hence
\begin{equation}
\begin{split}
\int_{D_{n,m,\varepsilon}}|\nabla(\theta_n-\theta_m)|^2 \dxdt \leq
2\varepsilon B+\varepsilon=(2 B+1 )\varepsilon.
\nonumber 
\end{split}
\end{equation}
The H\"older inequality yields
\begin{equation}
\begin{split}
\int_{D_{n,m,\varepsilon}}|\nabla(\theta_n-\theta_m)| \dxdt  & \leq
\left(\int_{D_{n,m,\varepsilon}}|\nabla(\theta_n-\theta_m)|^2 \dxdt \right)^{\frac{1}{2}}
\left(meas(D_{n,m,\varepsilon})\right)^{\frac{1}{2}}
\\
& \leq
C(2B+1)^{\frac{1}{2}}\varepsilon^{\frac{1}{2}}.
\nonumber 
\end{split}
\end{equation}
Using the decomposition of $Q=D_{n,m,\varepsilon}\cup (Q\setminus D_{n,m,\varepsilon})$ we have to consider  the integral over the second set.
\begin{equation}
\begin{split}
\int_{Q\setminus D_{n,m,\varepsilon}}|\nabla(\theta_n-\theta_m)|  \dxdt \leq
\left(\int_{Q\setminus D_{n,m,\varepsilon}}|\nabla(\theta_n-\theta_m)|^q \dxdt \right)^{\frac{1}{q}}
\left(meas(Q\setminus D_{n,m,\varepsilon})\right)^{1-\frac{1}{q}}
\end{split}
\end{equation}
The first term on the right hand side is bounded, since the sequence $\{\theta_n\}$ is uniformly bounded in $L^q(0,T,W^{1,q}(\Omega))$. The sequence $\{\theta_n\}$ is a Cauchy sequence in $L^1(0,T,L^1(\Omega))$, so there exists $n_0$ such that for all $n,m>n_0$ occurs $\left(meas(Q\setminus D_{n,m,\varepsilon})\right)^{1-\frac{1}{q}}<\varepsilon$. Then from the previous inequalities we obtain
\begin{equation}
\begin{split}
\int_Q|\nabla \theta_n-\nabla \theta_m| \dxdt & =
\int_{D_{n,m,\varepsilon}}|\nabla(\theta_n-\theta_m)| \dxdt  +
\int_{Q\setminus D_{n,m,\varepsilon}}|\nabla(\theta_n-\theta_m)| \dxdt  
\\
& \leq
c_1\varepsilon^{\frac{1}{2}}+c_2\varepsilon
\label{ciag_cau_szac}
\end{split}
\end{equation}
which implies that $\{\nabla\theta_n\}$ is a Cauchy sequence in $L^1(0,T,L^1(\Omega))$.

\end{proof}

\begin{lemat}{Aubin-Lions \cite[Lemma 7.7]{Roubicek}}

Let $V_1$, $V_2$ be Banach spaces, and $V_3$ be a metrizable Hausdorff locally convex space, 
$V_1$ be separable and reflexive, $V_1\subset\subset V_2$ (a compact embedding), $V_2 \subset V_3$ (a continuous embedding), $1 < p <\infty$,
$1 \leq q \leq \infty$. Then $\{u: u\in L^p(0,T,V_1);u_t\in L^q(0,T,V_3 )\}\subset\subset L^p(0,T,V_2 )$ (a compact embedding).
\end{lemat}

From the uniform boundedness of the sequence  $\{f_k\}$ in $L^1(0,T,L^1(\Omega))$  and from the uniform boundedness of the sequence $\{\theta_k\}$ in $L^q(0,T,W^{1,q}(\Omega))$ we obtain 
that $\{(\theta_n)_t\}$ is a sequence bounded in the space $L^1(0,T,W^{-1,q}(\Omega))$. 
Consequently the sequence $\{\theta_n\}$ is relatively compact in $L^1(0,T,L^1(\Omega))$. Due to Lemma \ref{unif_boun_LqLq} and Lemma \ref{zbieznosc_nabla_teta} we know that the sequence $\{\theta_n\}$ converges strongly to $\theta$ in $L^q(0,T,W^{1,q}(\Omega))$. 
Moreover, for $s$ large enough $(\theta_k)_t$ converges strongly to $\theta_t$ in $L^1(0, T; W^{-1,s}(\Omega))$. Thus, $\theta_k$ converges strongly to $\theta$ in $C([0, T], W^{-1,s}(\Omega))$ and $\theta_k(\cdot, 0)$ converges to $\theta(\cdot, 0)$ in $W^{-1,s}(\Omega)$.

\begin{lemat}
For $q<\frac{2(N+1) - N}{N + 1}$ ($q<\frac{5}{4}$ when $N=3$) there exists $\theta\in L^q(0,T,W^{1,q}(\Omega))\cap C([0,T],W^{-s,2}(\Omega))$ -  a solution to the system
\begin{equation}
\left\{
\begin{array}{cl}
\theta_t - \Delta\theta = f & \mbox{ in } \Omega\times (0,T),\\
\frac{\partial\theta}{\partial\vc{n}}=0 & \mbox{ on } \partial\Omega\times (0,T),\\
\theta (x,0) = \theta_0(x) & \mbox{ in } \Omega.
\end{array}
\right.
\end{equation}
\end{lemat}
\begin{proof}
Choosing in  \eqref{slab}  the test function $\psi\in C^{\infty}(\Omega\times [0,T))$ such that $\psi=0$ on $\Omega\times\{T\}$, we get
\begin{equation}
\begin{split}
\int_0^T\int_{\Omega}(\theta_n)_t\psi \dxdt - \int_0^T\int_{\Omega}\Delta\theta_n\psi \dxdt  &= \int_0^T\int_{\Omega}f_n\psi \dxdt .
\nonumber 
\end{split}
\end{equation}
Then
\begin{equation}
\begin{split}
-\int_0^T\int_{\Omega}\theta_n\psi_t  \dxdt & + 
\int_{\Omega}\theta_n\psi \dx \Big|_0^T 
\\
& + \int_0^T\int_{\Omega}\nabla\theta_n\cdot\nabla\psi  \dxdt -
\int_0^T\int_{\partial\Omega}\frac{\partial\theta_n}{\partial\vc{n}}\psi \dxdt 
 = \int_0^T\int_{\Omega}f_n\psi \dxdt .
\nonumber 
\end{split}
\end{equation}
And finally
\begin{equation}
\begin{split}
-\int_0^T\int_{\Omega}\theta_n\psi_t  \dxdt +
\int_0^T\int_{\Omega}\nabla\theta_n\cdot\nabla\psi  \dxdt 
& = \int_0^T\int_{\Omega}f_n\psi  \dxdt +
\int_{\Omega}T_n(\theta_0)\psi \dx .
\nonumber 
\end{split}
\end{equation}
Using the convergence of the temperatures' sequence we obtain
\begin{equation}
\begin{split}
-\int_0^T\int_{\Omega}\theta\psi_t +
\int_0^T\int_{\Omega}\nabla\theta\cdot\nabla\psi 
= \int_0^T\int_{\Omega}f\psi +
\int_{\Omega}\theta_0\psi.
\nonumber 
\end{split}
\end{equation}

\end{proof}

\section{}\label{B}
In the current section we present the construction of the basis used for approximation of the strain tensor. We adapt it for our particular case, however the idea follows the lines of  \cite[Theorem 4.11]{maleknecas}.
The definitions of spaces $V_k$ and $V_k^s$ were introduced in Section~\ref{sec:4} by \eqref{Vk} and \eqref{Vks}.


Let us consider the following problem: find $\ten{\zeta}_i\in V^s_k$ and $\lambda_i\in\mathbb{R}$ such that
\begin{equation}
\braket{\ten{\zeta}_i,\ten{\Phi}}_s = \lambda_i (\ten{\zeta}_i,\ten{\Phi})_{\ten{D}} \qquad \forall\ \ten{\Phi}\in V^s_k.
\label{eq:war_wl}
\end{equation}
where by $\braket{\cdot,\cdot}_s$ we denote the scalar product in $H^s(\Omega, \mathcal{S}^3)$ and $(\cdot,\cdot)_{\ten{D}}$ is the previously defined scalar product in $L^2(\Omega,\mathcal{S}^3)$.

\begin{tw}
There exist a countable set of eigenvalues $\{\lambda_i\}_{i=1}^{\infty}$ and a corresponding family of eigenfunctions $\{\ten{\zeta}_i\}_{i=1}^{\infty}$  solving \eqref{eq:war_wl} such that
\begin{itemize}
\item $(\ten{\zeta}_i,\ten{\zeta}_j)_{\ten{D}} = \delta_{ij}$ for all $i,j\in\mathbb{N}$,
\item $1\leq \lambda_1 \leq \lambda_2\leq ... $ and $\lambda_i\to \infty$ as $i$ tends to $\infty$,
\item $\braket{\frac{\ten{\zeta}_i}{\sqrt{\lambda_i}},\frac{\ten{\zeta}_j}{\sqrt{\lambda_i}}}_s = \delta_{ij}$ for all $i,j\in\mathbb{N}$,
\item the set $\{\ten{\zeta}_i\}_{i=1}^{\infty}$ is a basis of $V^s_k$.
\item the set $\{\ten{\zeta}_i\}_{i=1}^{\infty}$ is a basis of $V_k$.
\end{itemize}
Moreover, let us define the subspace $H^N\equiv \mbox{span} \{\ten{\zeta}_1,..., \ten{\zeta}_N\}$ and projection $P^N:V^s_k \to H^N$ such that $P^N(\ten{V}) \equiv \sum_{i=1}^N(\ten{V},\ten{\zeta}_i)_{\ten{D}}\ten{\zeta}_i$, then we get
\begin{equation}
\|P^N\varphi\|_{H^s}\le\|\varphi\|_{H^s}
\end{equation}
\label{th:jos}
\end{tw}

\begin{proof}
Proof of  Theorem \ref{th:jos} is divided into few steps.

\noindent
{\it Existence of $\ten{\zeta}_1$} 

\noindent
Let us define
\begin{equation}
\frac{1}{\lambda_1} \equiv \sup_{\ten{V}\in V^s_k \atop \|\ten{v}\|_{H^s}\leq 1} (\ten{V},\ten{V})_{\ten{D}}.
\label{eq:definiecja_war_wl}
\end{equation}
Consequently, there exists a sequence $\{\ten{V}_i\}_{i=1}^{\infty}$ such that $(\ten{V}_i,\ten{V}_i)_{\ten{D}}\to \frac{1}{\lambda_1}$ as $i$ tends to $\infty$ and $\|\ten{V}_i\|_{H^s(\Omega)}=1$. Then, there exist a subsequence $\{\ten{V}_i\}_{i=1}^{\infty}$ (still denoted by i) and  $\ten{\zeta}_1 \in V^s_k$ such that
\begin{equation}
\begin{split}
\ten{V}_i & \rightharpoonup \ten{\zeta}_1
\qquad \mbox{in } V^s_k,
\\
\ten{V}_i & \to \ten{\zeta}_1
\qquad \mbox{in } L^2(\Omega,\mathcal{S}^3).
\end{split}
\end{equation}
If $\|\ten{\zeta}_1\|_{H^s(\Omega)} < 1$, then let us define $\ten{\zeta}=\frac{\ten{\zeta}_1}{\|\ten{\zeta}_1\|_{H^s(\Omega)}}$ and then
\begin{equation}
\|\ten{\zeta}\|_{H^s(\Omega)} =1
\qquad
\mbox{ and }
\qquad
(\ten{\zeta},\ten{\zeta})_{\ten{D}} = \frac{(\ten{\zeta}_1,\ten{\zeta}_1)_{\ten{D}}}{\|\ten{\zeta}_1\|_{H^s(\Omega)}} > \frac{1}{\lambda_1},
\end{equation}
which is contrary with \eqref{eq:definiecja_war_wl} and it implies that $\|\ten{\zeta}_1\|_{H^s(\Omega)} = 1$. To finish the first step we show that $\ten{\zeta}_1$ is an eigenfunction. Let us take $\ten{H}\in V^s_k$ and define the function
\begin{equation}
\Phi(t) = \frac{(\ten{\zeta}_1 + t\ten{H},\ten{\zeta}_1 + t\ten{H})_{\ten{D}}}{\braket{\ten{\zeta}_1 + t\ten{H},\ten{\zeta}_1 + t\ten{H}}_s}.
\end{equation}
Calculating the derivative of function $\Phi(t)$, we obtain
\begin{equation}
\begin{split}
0 = \frac{d}{dt} \Phi(t) |_{t=0} & =
\frac{2(\ten{\zeta}_1,\ten{H})_{\ten{D}}\braket{\ten{\zeta}_1 ,\ten{\zeta}_1}_s - 2(\ten{\zeta}_1,\ten{\zeta}_1)_{\ten{D}}\braket{\ten{\zeta}_1,\ten{H}}_s}{((\ten{\zeta}_1 ,\ten{\zeta}_1 ))_s^2}
\\
& =
\frac{2(\ten{\zeta}_1,\ten{H})_{\ten{D}} - \frac{2}{\lambda_1}\braket{\ten{\zeta}_1,\ten{H}}_s}{\braket{\ten{\zeta}_1 ,\ten{\zeta}_1 }_s^2}
\end{split}
\end{equation}
and then
\begin{equation}
\lambda_1(\ten{\zeta}_1,\ten{H})_{\ten{D}} = \braket{\ten{\zeta}_1,\ten{H}}_s
\qquad
\forall\ \ten{H}\in V^s_k.
\end{equation}

\noindent
{\it Iterative construction} 

\noindent
Assume that for $N\geq 1$ there exists the set of eigenvalues $\{\lambda_i\}_{i=1}^{N}$ and the set of corresponding eigenfunctions $\{\ten{\zeta}_i\}_{i=1}^{N}$. Let us define the space
\begin{equation}
W^N \equiv \{ \ten{V}\in V^s_k: \braket{\ten{V},\ten{\zeta}_i}_s =0, \quad i=1,...,N\}.
\end{equation}
Using the similar construction as in the previous step, we find the next eigenvalue and eigenfunction
\begin{equation}
(\ten{\zeta}_{N+1},\ten{\zeta}_{N+1})_{\ten{D}} = \sup_{\ten{V}\in W^N\atop \|\ten{V}\|_{H^s}=1} (\ten{V},\ten{V})_{\ten{D}} \equiv \frac{1}{\lambda_{N+1}}.
\end{equation}

Finally, we obtain 
\begin{equation}
\begin{split}
&1\leq \lambda_1\leq \lambda_2\leq ...,
\\
&(\ten{\zeta}_i,\ten{\zeta}_j)_{\ten{D}}=0 \qquad \mbox{if } i\leq j,
\\
&\braket{\ten{\zeta}_i,\ten{\zeta}_j}_s=\delta_{ij}.
\end{split}
\end{equation}

\noindent
{\it Unboundedness of eigenvalues} 

\noindent
Let us assume that the set of eigenvalues has  a finite limit, i.e. $\lim_{i\to\infty}\lambda_i =\lambda <\infty$. Since $\|\ten{W}_i\|_{H^s}=1$, using subsequence if it is necessary, we get $\ten{W}_i \to \ten{W}$ in $L^2(\Omega,\mathcal{S}^3)$ as $i\to\infty$. Hence
\begin{equation}
\begin{split}
2 &= \braket{\ten{\zeta}_i,\ten{\zeta}_i}_s + \braket{\ten{\zeta}_j,\ten{\zeta}_j}_s
= \braket{\ten{\zeta}_i - \ten{\zeta}_j,\ten{\zeta}_i - \ten{\zeta}_j}_s
\\
& = \braket{\ten{\zeta}_i ,\ten{\zeta}_i - \ten{\zeta}_j}_s
- \braket{\ten{\zeta}_j,\ten{\zeta}_i - \ten{\zeta}_j}_s
\\
& = \lambda_i (\ten{\zeta}_i ,\ten{\zeta}_i - \ten{\zeta}_j)_{\ten{D}}
- \lambda_j(\ten{\zeta}_j,\ten{\zeta}_i - \ten{\zeta}_j)_{\ten{D}}.
\end{split}
\label{eq:sprz}
\end{equation}
Passing with $i,j$ to $\infty$ we obtain
\begin{equation}
\begin{split}
(\ten{\zeta}_i ,\ten{\zeta}_i - \ten{\zeta}_j)_{\ten{D}} \to 0,
\\
(\ten{\zeta}_j,\ten{\zeta}_i - \ten{\zeta}_j)_{\ten{D}} \to 0.
\end{split}
\label{eq:sprz2}
\end{equation}
Comparing \eqref{eq:sprz} and \eqref{eq:sprz2} we get the contradiction.

\noindent
{\it The Set $\{\lambda_i\}_{i=1}^{\infty}$ contains all eigenvalues} \\
Let us assume that there exists an eigenvalue $\lambda$ such that $\lambda \notin \{\lambda_i\}_{i=1}^{\infty}$. Let $\ten{W}$ be the corresponding eigenfunction to the eigenvalue $\lambda$ and
\begin{equation}
\braket{\ten{\zeta},\ten{\Phi} }_s=\lambda (\ten{\zeta},\ten{\Phi})_{\ten{D}}
\qquad
\ten{\Phi} \in V^s_k.
\end{equation}
Without loss of  generality, $\|\ten{\zeta}\|_{H^s}=1$.
Moreover, there exists $i\in\mathbb{N}$ such that $\lambda_i < \lambda <\lambda_{i+1}$. Then, for all $k=1,...,i$
\begin{equation}
\begin{split}
\braket{\ten{\zeta}_k,\ten{\zeta} }_s = \lambda_k (\ten{\zeta}_k,\ten{\zeta} )_{\ten{D}},
\\
\braket{\ten{\zeta},\ten{\zeta}_k }_s = \lambda (\ten{\zeta},\ten{\zeta}_k )_{\ten{D}}.
\end{split}
\end{equation}
Hence, $(\ten{\zeta},\ten{\zeta}_k )_{\ten{D}}=0$ and therefore $\ten{\zeta}\in W^i$ and
\begin{equation}
(\ten{\zeta},\ten{\zeta})_{\ten{D}}=\frac{1}{\lambda}> \frac{1}{\lambda_i} = \sup_{\ten{V}\in W^N \atop \|\ten{V}\|_{s,2}=1} (\ten{V},\ten{V})_{\ten{D}},
\end{equation}
which is a contradiction.

\noindent
{\it The set $\{\ten{\zeta}_i\}_{i=1}^{\infty}$ is a basis in $V^s_k$} 

\noindent
Let us define $X= \mbox{span}\{\ten{\zeta}_1,\ten{\zeta}_2,...\}$ and let us assume that $X\neq V^s_k$. Then, there exists $\ten{\Phi}\in V^s_k$ such that $\|\ten{\Phi}\|_{H^s(\Omega)} =1$ and $\braket{\ten{\Phi},\ten{\zeta}_i}_{s}=0$ for all $i\in\mathbb{N}$. Moreover, for all $i\in\mathbb{N}$
\begin{equation}
(\ten{\Phi},\ten{\Phi})_{\ten{D}} \leq \sup_{\ten{V}\in W^N \atop \|\ten{V}\|_{H^s}=1} (\ten{V},\ten{V})_{\ten{D}} = \frac{1}{\lambda_1},
\end{equation}
which implies that $\ten{\Phi}=\ten{0}$.

\noindent
{\it Renormalization of basis} 

\noindent
To complete the proof we may renormalize the basis 
\begin{equation}
\widehat{\ten{\zeta}_i}\equiv \frac{\ten{\zeta}_i}{\sqrt{\lambda_i}}.
\end{equation}
for all $i\in\mathbb{N}$.

\noindent
{\it The set $\{\ten{\zeta}_i\}_{i=1}^{\infty}$ is a basis in $V_k$} 

\noindent
Observe that the space $V_k^s$ is dense in $V_k$ in $L^2(\Omega, \mathcal{S}^3)$ norm. For this purpose consider an element $\ten{\xi}$ of $V^k$. To show there exists a sequence $\ten{\xi}^n$  bounded in $V^k_s$ that converges to $\ten{\xi}$ recall that  if $\ten{\xi}$ is in  $L^2(\Omega,\mathcal{S}^3)$, then there exists an approximating sequence
$\overline{\ten{\xi}}^n$ in $H^s(\Omega,\mathcal{S}^3)$. Then the sequence $\ten{\xi}^n$ we construct as follows
$$\ten{\xi}^n:=\overline{\ten{\xi}}^n-P_k\overline{\ten{\xi}}^n,$$
where the projection $P^k$ was defined in the proof of Lemma~\ref{wsp_org_epa}. Then using the continuity of $P^k$ in $H^s(\Omega, \mathcal{S}^3)$ we immediately obtain that $\ten{\xi}^n$ is bounded in $H^s(\Omega, \mathcal{S}^3)$ and converges to $\ten{\xi}\in V_k$. Consequently, $\{\ten{\zeta}_i\}_{i=1}^{\infty}$ is also a basis in $V_k$.

\end{proof}

\end{appendix}

%

\bigskip

{\bf Acknowledgement} P.G. is a coordinator and F.K. is a PhD student in the International PhD Projects Programme of Foundation for Polish Science operated within the Innovative Economy Operational Programme 2007-2013 funded by European Regional Development Fund (PhD Programme: Mathematical Methods in Natural Sciences). PG was supported by the National Science Center, project no. 6085/B/H03/2011/40.  FK was partially supported by grant NCN OPUS 2012/07/B/ST1/03306, A. \'S.-G. was supported  by the grant IdP2011/000661.


\begin{thebibliography}{10}

\bibitem{Alber}
H.-D. Alber.
\newblock {\em Materials with memory}.
\newblock Lecture Notes in Math. 1682. Springer, Berlin Heidelberg, New York,
  1998.

\bibitem{CheAl}
H.-D. Alber and K.~Che{\l}mi\'{n}ski.
\newblock Quasistatic problems in viscoplasticity theory ii: Models with
  nonlinear hardening.
\newblock {\em Mathematical Models and Methods in Applied Sciences},
  17(02):189--213, 2007.

\bibitem{Barbu}
V.~Barbu.
\newblock {\em Nonlinear semigroups and differential equations in Banach
  spaces}.
\newblock Noordhoff International Publishing, 1976.

\bibitem{Bartczak}
L.~Bartczak.
\newblock Mathematical analysis of a thermo-visco-plastic model with
  {B}odner-{P}artom constitutive equations.
\newblock {\em Journal of Mathematical Analysis and Applications}, 385(2):961
  -- 974, 2012.

\bibitem{BR}
S.~Bartels and T.~Roub\'i\v{c}ek.
\newblock Thermo-visco-elasticity with rate-independent plasticity in isotropic
  materials undergoing thermal expansion.
\newblock {\em ESAIM Math. Model. Numer. Anal.}, 45:477--504, 2011.

\bibitem{Blanchard}
D.~Blanchard.
\newblock Truncations and monotonicity methods for parabolic equations.
\newblock {\em Nonlinear Analysis: Theory, Methods and Applications},
  21(10):725 -- 743, 1993.

\bibitem{BlanchardMurat}
D.~Blanchard and F.~Murat.
\newblock Renormalised solutions of nonlinear parabolic problems with l1 data:
  existence and uniqueness.
\newblock {\em Proceedings of the Royal Society of Edinburgh: Section A
  Mathematics}, 127:1137--1152.

\bibitem{Boccardo}
L.~Boccardo and T.~Gallouet.
\newblock Non-linear elliptic and parabolic equations involving measure data.
\newblock {\em Journal of Functional Analysis}, 87(1):149 -- 169, 1989.

\bibitem{brokate}
M.~Brokate, P.~Krej\v{c}\'{i}, and D.~Rachinskii.
\newblock Some analytical properties of the multidimensional continuous
  {M}r\'{o}z model of plasticity. recent advances in structural modelling and
  optimization.
\newblock {\em Control Cybernet.}, 27(2):199 -- 215, 1998.

\bibitem{PhDBulicek}
M.~Bul\'\i\v{c}ek.
\newblock Navier's slip and evolutionary {N}avier-{S}tokes-{F}ourier-like
  systems with pressure, shear-rate and temperature dependent viscosity.
\newblock {\em PhD thesis}, 2006.

\bibitem{BFM}
M.~Bul\'\i\v{c}ek, E.~Feireisl, and J.~M\'{a}lek.
\newblock A {N}avier-{S}tokes-{F}ourier system for incompressible fluids with
  temperature dependent material coefficients.
\newblock {\em Nonlinear Analysis: Real World Applications}, 10(2):992 -- 1015,
  2009.

\bibitem{Bull}
M.~Bul\'\i\v{c}ek and P.~Pust\v{e}jovsk\'a.
\newblock On existence analysis of steady flows of generalized {N}ewtonian
  fluids with concentration dependent power-law index.
\newblock {\em Journal of Mathematical Analysis and Applications}, 402(1):157
  -- 166, 2013.

\bibitem{MMA:MMA802}
K.~Che{\l}mi{\'n}ski.
\newblock On large solutions for the quasistatic problem in non-linear
  viscoelasticity with the constitutive equations of {B}odner--{P}artom.
\newblock {\em Mathematical Methods in the Applied Sciences}, 19(12):933--942,
  1996.

\bibitem{CHelG2}
K.~Che{\l}mi{\'n}ski and P.~Gwiazda.
\newblock Nonhomogeneous initial-boundary value problems for coercive and
  self-controlling models of monotone type.
\newblock {\em Continuum Mechanics and Thermodynamics}, 12:217--234, 2000.

\bibitem{MANA:MANA5}
K.~Che{\l}mi{\'n}ski and P.~Gwiazda.
\newblock On the model of {B}odner -- {P}artom with nonhomogeneous boundary
  data.
\newblock {\em Mathematische Nachrichten}, 214(1):5--23, 2000.

\bibitem{ChGw2007}
K.~Che{\l}mi{{\'n}}ski and P.~Gwiazda.
\newblock Convergence of coercive approximations for strictly monotone
  quasistatic models in inelastic deformation theory.
\newblock {\em Math. Methods Appl. Sci.}, 30(12):1357--1374, 2007.

\bibitem{MMA:MMA844}
K.~Che{\l}mi{\'n}ski and P.~Gwiazda.
\newblock Convergence of coercive approximations for strictly monotone
  quasistatic models in inelastic deformation theory.
\newblock {\em Mathematical Methods in the Applied Sciences},
  30(12):1357--1374, 2007.

\bibitem{ChR}
K.~Che{\l}mi{\'n}ski and R.~Racke.
\newblock Mathematical analysis of a model from thermoplasticity with kinematic
  hardening.
\newblock {\em J. Appl. Anal.}, 12:37--57, 2006.

\bibitem{1240/THESES}
S.~Clain.
\newblock {\em Analyse math{\'e}matique et num{\'e}rique d'un mod{\`e}le de
  chauffage par induction}.
\newblock PhD thesis, Lausanne, 1994.

\bibitem{duvautLions}
G.~Duvaut and J.L. Lions.
\newblock {\em Les in\'equations en m\'ecanique et en physique}.
\newblock Dunod, Paris, 1972.

\bibitem{Evans}
L.~C. Evans.
\newblock {\em Partial Differential Equations}.
\newblock American Math Society, 1998.

\bibitem{GreenNaghdi}
A.~E. Green and P.~M. Naghdi.
\newblock A general theory of an elastic-plastic continuum.
\newblock {\em Archive for Rational Mechanics and Analysis}, 18:251--281, 1965.

\bibitem{Gurtin}
M.~E. Gurtin, E.~Fried, and L.~Anand.
\newblock {\em The Mechanics and Thermodynamics of Continua}.
\newblock Cambridge, 2013.

\bibitem{GKSG}
P.~Gwiazda, F.Z. Klawe, and A.~\'Swierczewska-Gwiazda.
\newblock Thermo-visco-elasticity for the {M}r\'{o}z's model in the framework
  of thermodynamically complete systems.
\newblock {\em to appear in Discrete Contin. Dyn. Syst. Ser. S}.

\bibitem{GwSw2005}
P.~Gwiazda and A.~{{\'S}}wierczewska.
\newblock Large eddy simulation turbulence model with {Y}oung measures.
\newblock {\em Appl. Math. Lett.}, 18(8):923--929, 2005.

\bibitem{Homberg200455}
Dietmar H\"{o}mberg.
\newblock A mathematical model for induction hardening including mechanical
  effects.
\newblock {\em Nonlinear Analysis: Real World Applications}, 5(1):55 -- 90,
  2004.

\bibitem{JR}
S.~Jiang and R.~Racke.
\newblock {\em Evolution equations in thermoelasticity}.
\newblock Chapman \& Hall/CRC, Boca Raton. 2000.

\bibitem{johnson1}
C.~Johnson.
\newblock Existence theorems for plasticity problem.
\newblock {\em J. Math. Pures Appl.}, 55:431--444, 1976.

\bibitem{johnson2}
C.~Johnson.
\newblock On plasticity with hardening.
\newblock {\em J. Math. Anal. Appl.}, 62:325--336, 1978.

\bibitem{LandauLifshitz}
L.D. Landau and E.M. Lifshitz.
\newblock {\em Theory of Elasticity}.
\newblock Pergamon Press, 1970 (7th edition).

\bibitem{maleknecas}
J.~M\'alek, J.~Ne\v{c}as, M.~Rokyta, and M.~R\r{u}\v{z}i\v{c}ka.
\newblock {\em Weak and measure-valued solutions to evolutionary PDEs}.
\newblock Chapman \& Hall, London, 1996.

\bibitem{NH}
J.~Ne\v{c}as and I.~Hlav\'{a}\v{c}ek.
\newblock {\em Mathematical theory of elastic and elasto-plastic bodies: an
  introduction}.
\newblock Elsevier Scientific Publisher Company, 1980.

\bibitem{rajagopal}
K.~R. Rajagopal.
\newblock {\em Implicit constitutive relations}.
\newblock Encyclopedia of Life Support Systems (EOLSS), Developed under the
  Auspices of the UNESCO, Eolss Publishers, Oxford, UK.

\bibitem{Roubicek}
T.~Roub\'i\v{c}ek.
\newblock {\em Nonlinear Partial Differential Equations with Applications}.
\newblock Birkhauser Verlag, 2005.

\bibitem{suquet1}
P.~Suquet.
\newblock {\em Existence and regularity of solution for plasticity problems}.
\newblock in Variational Methods in Solid Mechanics, Pergamon Press, Oxford,
  1980.

\bibitem{suquet2}
P.~Suquet.
\newblock Sur les \'equations de la plasticit\'e: existence et r\'egularit\'e
  des solutions.
\newblock {\em J. M\'ecanique}, 20:3--39, 1981.

\bibitem{suquet3}
P.~Suquet.
\newblock {\em Plasticit\'e et homog\'en\'eisation}.
\newblock PhD thesis, 1982.

\bibitem{Sw2006}
A.~{{\'S}}wierczewska.
\newblock A dynamical approach to large eddy simulation of turbulent flows:
  existence of weak solutions.
\newblock {\em Math. Methods Appl. Sci.}, 29(1):99--121, 2006.

\bibitem{temam1}
R.~Temam.
\newblock {\em Mathematical Problems in Plasticity}.
\newblock Gauthier-Villars, Paris-New York, 1984.

\bibitem{temam2}
R.~Temam.
\newblock A generalized {N}orton-{H}off model and the {P}randtl-{R}euss law of
  plasticity.
\newblock {\em Archive for Rational Mechanics and Analysis}, 95(2):137--183,
  1986.

\bibitem{temammiranville}
R.~M. Temam and A.~M. Miranville.
\newblock {\em Mathematical modeling in continuum mechanics}.
\newblock Cambridge University Press, 2005.

\bibitem{Valent}
T.~Valent.
\newblock {\em Boundary Value Problems of Finite Elasticity: Local Theorems on
  Existence, Uniqueness, and Analytic Dependence on Data}.
\newblock Springer Publishing Company, Incorporated, 1st edition, 2011.

\bibitem{zeidlerB}
E.~Zeidler.
\newblock {\em Nonlinear Functional Analysis II/B -- Nonlinear Monotone
  Operators}.
\newblock Springer--Verlag, Berlin--Heidelberg--New York, 1990.

\end{thebibliography}
\bibliographystyle{plain}

\end{document}